\documentclass[10pt,reqno]{amsart}
\usepackage{amsmath, amssymb, amsthm, epsfig}
\usepackage{amsfonts}
\usepackage{mathrsfs}

\usepackage{graphicx}
\usepackage{verbatim}
\usepackage{color}

\newtheorem{lemma}{Lemma}[section]
\newtheorem{theorem}{Theorem}[section]
\newtheorem{proposition}{Proposition}[section]

\theoremstyle{remark}
\newtheorem{rem}{Remark}[section]

\numberwithin{equation}{section}

\renewcommand{\r}{\rho}
\newcommand{\R}{{\mathbb R}}

\newcommand{\p}{\partial}

\allowdisplaybreaks

\newcommand{\norm}[1]{\left\Vert#1\right\Vert}
\newcommand{\abs}[1]{\left\vert#1\right\vert}
\newcommand{\set}[1]{\left\{#1\right\}}

\newcommand{\pr}[1]{\left(#1\right)}

\newcommand{\lb}[1]{\left\{#1\right.}

\newcommand{\defs}{:=}
\newcommand{\sTo}{\rightarrow}

\newcommand{\pt}[1]{\partial_{#1}}
\newcommand{\mS}{\mathcal{S}}

\DeclareMathOperator{\tg}{tan} 
\DeclareMathOperator{\ctg}{cot}

\newcommand{\Real}{\mathbb R}

\newcommand{\D}{\mathscr{D}}

\newcommand{\G}{\mathcal{G}}

\newcommand{\Pe}{\mathcal{P}}

\newcommand{\me}{\mathrm{e}}
\newcommand{\dif}{\mathrm{d}}
\newcommand{\mi}{\mathrm{i}}

\newcommand{\fa}{\mathscr{A}}

\newcommand{\ff}{\mathscr{F}}
\newcommand{\fj}{\mathscr{J}}

\newcommand{\ft}{\mathscr{T}}
\newcommand{\fp}{\mathscr{P}}

\begin{document}
\title[Stability of Transonic Shocks in Potential Flow]
{Stability of Transonic Shock-Fronts in Three-Dimensional Conical
Steady Potential Flow past a Perturbed Cone}

\author{Gui-Qiang Chen \and Beixiang Fang}
\address{Gui-Qiang Chen:
School of Mathematical Sciences, Fudan University, Shanghai 200433, China;\\
Department of Mathematics, Northwestern University,
2033 Sheridan Road, Evanston, IL 60208-2730, USA\\
http://www.math.northwestern.edu/\~{}gqchen} \email{\tt
gqchen@math.northwestern.edu}
\address{Beixiang Fang: Department of Mathematics, Shanghai Jiaotong University,
                Shanghai 200240, China}
\email{\tt bxfang@gmail.com}

\keywords{}
\subjclass[2000]{35L65,35L67,35M10,35B35,76H05,76N10}
\date{\today}
\thanks{}
\dedicatory{Dedicated to Professor Tatsien Li on the Occasion of His
70th Birthday}

\begin{abstract}
For an upstream supersonic flow past a straight-sided cone in $\R^3$
whose vertex angle is less than the critical angle, a transonic
(supersonic-subsonic) shock-front attached to the cone vertex can be
formed in the flow.
In this paper we analyze the stability of transonic shock-fronts in
three-dimensional steady potential flow past a perturbed cone. We
establish that the self-similar transonic shock-front solution is
conditionally stable in structure with respect to the conical
perturbation of the cone boundary and  the upstream flow in
appropriate function spaces.
In particular, it is proved that
the slope of the
shock-front tends asymptotically to the slope of the unperturbed
self-similar shock-front downstream at infinity.

In order to achieve these results, we first formulate the stability
problem as a free boundary problem and then introduce a coordinate
transformation to reduce the free boundary problem into a fixed
boundary value problem for a singular nonlinear elliptic system. We
develop an iteration scheme that consists of two iteration mappings:
one is for an iteration of approximate transonic shock-fronts; and
the other is for an iteration of the corresponding boundary value
problems of the singular nonlinear systems for the given approximate
shock-fronts. To ensure the well-definedness and contraction
property of the iteration mappings, we develop an approach to
establish the well-posedness for a corresponding singular linearized
elliptic equation, especially the stability with respect to the
coefficients of the elliptic equation, and to obtain the estimates
of its solutions reflecting both their singularity at the cone
vertex and decay at infinity. The approach is to employ key features
of the equation, introduce appropriate solution spaces, and apply a
Fredholm-type theorem to establish the existence of solutions by
showing the uniqueness in the solution spaces.
\end{abstract}
\maketitle

\section{Introduction}

We study the stability of transonic shock-fronts in
three-dimensional steady potential flow past a perturbed cone.
The steady potential equations with cylindrical symmetry with
respect to the $x$-axis can written as
\begin{equation}\label{PEq}
\lb{\begin{aligned}&\pt{x}(\rho u) + \pt{y}(\rho v) + \frac{\rho
v}{y} = 0,\\&\pt{x}v - \pt{y}u = 0,
\end{aligned}}
\end{equation}
together with Bernoulli's law:
\begin{equation}\label{BL}
   \frac 12 (u^2+v^2) + \frac{1}{\gamma - 1}\rho^{\gamma-1} =
   \kappa_\infty,
\end{equation}
where $\kappa_\infty:=\frac12u_\infty^2+
\frac1{\gamma-1}\rho_\infty^{\gamma-1}$ is determined by the
upstream flow state at infinity, i.e., the density $\rho_\infty$ and
velocity $(u_\infty,0)$, and $y$ is the distance of the flow
location in $\R^3$ to the $x$-axis. In \eqref{BL}, we have used the
pressure-density relation:
\begin{equation}\label{PGas}
p=\frac{\rho^\gamma}{\gamma}, \qquad \gamma>1,
\end{equation}
so that the sound speed $c=\rho^{(\gamma-1)/2}$.

For an upstream supersonic flow past a straight-sided cone, a
shock-front is formed in the flow. When the vertex angle of the cone
is less than the critical angle, the shock-front may be self-similar
and attached to the cone vertex. There are two kinds of admissible
shock-fronts depending on the downstream condition at infinity (cf.
Courant-Friedrichs \cite{CoF}, Chapter VI): transonic
(supersonic-subsonic) shock-fronts and supersonic-supersonic
shock-fronts. In this paper, we are interested in the stability of
the transonic shock-front, behind which the flow is completely
subsonic (see Fig. \ref{fig:1}).
%
%
%
%
More precisely, for fixed upstream density $\rho_\infty>0$ at
infinity, our problem is to understand the stability of self-similar
transonic shock-front when the speed of the upstream flow velocity
$(u_\infty,0)$ is large, equivalently, when the Mach number
$M_\infty:=\frac{u_\infty}{c_\infty}$ is large.

\begin{figure}[htbp]
\begin{center}
\includegraphics[height=1.9in,width=3in]{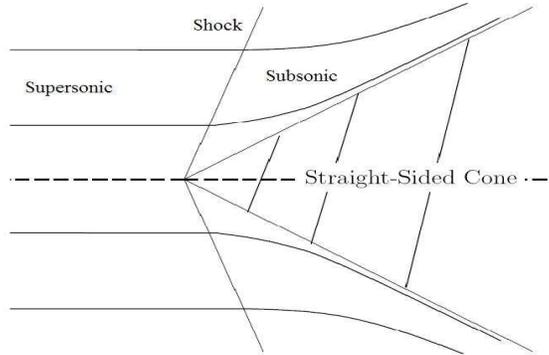}
\caption{A self-similar transonic shock in three-dimensional steady
flow past a straight-sided cone} \label{fig:1}
\end{center}
\end{figure}

By scaling the state variables $(u,v,\rho)\to (\tilde{u}, \tilde{v},
\tilde{\rho})$:
\begin{equation}\label{scaling}
(\tilde{u}, \tilde{v}, \tilde{\rho})=(\frac{u}{u_\infty},
\frac{v}{u_\infty}, \frac{\rho}{u_\infty^{2/(\gamma-1)}}),
%
\end{equation}
the corresponding sound speed becomes
$\tilde{c}=\frac{c}{u_\infty}$; the equations in \eqref{PEq} remain
unchanged for the new variables $(\tilde u, \tilde v, \tilde\rho)$,
and the Bernoulli constant becomes $\tilde\kappa_\infty:=\frac12+
\frac1{\gamma-1}\tilde\rho_\infty^{\gamma-1}$. Therefore, without
loss of generality, we can drop `` $\tilde{}\ $'' for notational
convenience hereafter to assume that $u_\infty=1$, the Bernoulli
constant is
\begin{equation}\label{BLT}
\kappa_\infty:=\frac12+ \frac1{\gamma-1}\rho_\infty^{\gamma-1}.
\end{equation}
%
Then we have
\begin{equation}\label{MachUp}
M_\infty^2=\rho_\infty^{-(\gamma-1)}\quad \text{ or }\quad \nu\defs
c_\infty^2=\frac1{M_\infty^2}=\rho_\infty^{\gamma-1}.
\end{equation}
Under this scaling, the problem reduces to the stability problem for
self-similar transonic shock-fronts in transonic flow past a
perturbed cone, governed by \eqref{PEq}--\eqref{BL} with the
Bernoulli constant \eqref{BLT}, when the Mach number $M_\infty$ of
the upcoming flow is sufficiently large, or equivalently, the
density $\rho_\infty$ is sufficiently small.

\smallskip Conical flow (i.e. cylindrically symmetric flow
with respect to an axis, say, the $x$-axis) occurs in many physical
situations. For instance, it occurs at the conical nose of a
projectile facing a supersonic stream of air (cf. \cite{CoF}).
The study of supersonic-supersonic shock-fronts was initiated in Gu
\cite{Gu}, Schaeffer \cite{Schaeffer}, and Li \cite{Li} first for
the wedge case; also see Chen \cite{Chen2,Chen3,Sxchen1}, Zhang
\cite{Zh1,Zh2}, and Chen-Zhang-Zhu \cite{ChenZhangZhu} for the
recent results. The stability of conical supersonic-supersonic
shock-fronts has been studied in the recent years in Liu-Lien
\cite{LL}
in the class of $BV$ solutions when the cone vertex angle is small,
and Chen \cite{Sxchen2} and Chen-Xin-Yin \cite{CXY} in the class of
smooth solutions away from the conical shock-front when the
perturbed cone is sufficiently close to the straight-sided cone.

The stability of transonic shock-fronts in three-dimensional steady
flow past a perturbed cone has been a longstanding open problem.
Some progress has been made for the wedge case in two-dimensional
steady flow in Chen-Fang \cite{ChenFang} and Fang \cite{Fang}. In
particular, in \cite{ChenFang,Fang}, it was proved that the
transonic shock is conditionally stable under perturbation of the
upstream flow and/or perturbation of wedge boundary. Also see
\cite{CCF,ChenFeldman1,ChenFeldman2,Sxchen3,XY,Yuan1,Yuan2} for
steady transonic flow in multidimensional nozzles.

For the two-dimensional wedge case, the equations do not involve
singular terms and the flow past the straight-sided wedge is
piecewise constant. However, for the three-dimensional conical case,
the governing equations have a singularity at the cone vertex and
the flow past the straight-sided cone is self-similar, but is no
longer piecewise constant.
These cause additional difficulties for the stability problem. In
this paper, we develop techniques to handle the singular terms in
the equations and the singularity of the solutions.

Our main results indicate that the self-similar transonic
shock-front is conditionally stable with respect to the conical
perturbation of the cone boundary and the upstream flow in
appropriate function spaces.
That is, it is proved that the transonic shock-front and downstream
flow in our solutions are close to the unperturbed self-similar
transonic shock-front and downstream flow under the conical
perturbation, and
the slope
of the shock-front asymptotically tends to the slope of the
unperturbed self-similar shock at infinity.

In order to achieve these results, we first formulate the stability
problem as a free boundary problem and then introduce a coordinate
transformation to reduce the free boundary problem into a fixed
boundary value problem for a singular nonlinear elliptic system. We
develop an iteration scheme that consists of two iteration mappings:
one is for an iteration of approximate transonic shock-fronts; and
the other is for an iteration of the corresponding boundary value
problems for the singular nonlinear systems for given approximate
shock-fronts. To ensure the well-definedness and contraction
property of the iteration mappings, it is essential to establish the
well-posedness for a corresponding singular linearized elliptic
equation, especially the stability with respect to the coefficients
of the equation,  and obtain the estimates of its solutions
reflecting their singularity at the cone vertex and decay at
infinity. The approach is to employ key features of the equation,
introduce appropriate solution spaces, and apply a Fredholm-type
theorem in Maz'ya-Plamenevski\v{\i} \cite{MP} to establish the
existence of solutions by showing the uniqueness in the solution
spaces.
%
%

The organization of this paper is as follows. In Section 2, we
exploit the behavior of self-similar transonic shocks and
corresponding transonic flows past straight-sided cones, governed by
\eqref{PEq}--\eqref{BL} with Bernoulli constant \eqref{BLT}. In
Section 3, we first formulate the stability problem as a free
boundary problem, then introduce a coordinate transformation to
reduce the free boundary problem into a fixed boundary value
problem, and finally state the main theorem (Theorem 3.1) of this
paper and its equivalent theorem (Theorem 3.2).

In Section 4, we establish the well-posedness for a singular linear
elliptic equation, which will play an important role for
establishing the main theorem, Theorem \ref{th:FixB}. In Section 5,
we develop our iteration scheme for the stability problem, which
includes two steps: one is an iteration of approximate transonic
shock-fronts; and the other is the iteration of the corresponding
nonlinear boundary value problems for given approximate
shock-fronts. In Sections 6--7, we prove that the two iteration
mappings in the iteration scheme are both well-defined, contraction
mappings, based on the well-posedness theory for a singular linear
elliptic equation established in Section 4. This implies that there
exists a unique fixed point of each iteration mapping leading to the
completion of the proof of the main theorem, Theorem 3.1.

We remark that all the results for the case $\gamma>1$ is valid for
the isothermal case $\gamma=1$ as the limiting case when $\gamma\to
1$, which can be checked step by step in the proofs.

\section{Self-similar transonic shocks and corresponding transonic flows
          past straight-sided cones}

In this section, we  exploit the behavior of self-similar transonic
shocks and corresponding transonic flows past straight-sided cones,
governed by \eqref{PEq}--\eqref{BL} with Bernoulli constant
\eqref{BLT}.

Let the turning angle of the velocity field right behind the
self-similar shock-front $\mathcal{S}$ be $\phi_1$ and set
$b=\tg\phi_1$. Then $v=bu$ for the velocity field $(u,v)$ of the
flow right across $\mathcal{S}$. Assume that the angle between $\mS$
and the upcoming velocity field $(1,0)$ is $\omega_1$ and set
$\tau=\ctg\omega_1$. Then the Rankine-Hugoniot conditions on $\mS$
are
\begin{equation}\label{BgS04}
[\rho u]=\tau[\rho v],\qquad -[v] =\tau[u].
\end{equation}
Using \eqref{BgS04} and the relation $v=bu$, we have
\begin{equation}\label{BgS05}
u=\frac{\tau }{b+\tau},\qquad v=\frac{b\tau }{b+\tau},\qquad
\rho=\frac{b+\tau}{\tau(1-b\tau)}\rho_\infty.
\end{equation}

Substitute \eqref{BgS05} into Bernoulli's law with Bernoulli
constant \eqref{BLT}
and use $\nu=\frac1{M_\infty^2}$. Then a direct computation yields
\begin{equation}\label{BgS06}
0=F(\tau,\nu):=\tau - \frac{b+\tau}{1-b\tau}\Big(
\frac{(\gamma-1)(1+2\tau/b-\tau^2)}{2(1+\tau/b)^2} +
\nu\Big)^{-\frac{1}{\gamma-1}}\nu^{\frac{1}{\gamma-1}}.
\end{equation}

For $\gamma>1$ and $b>0$, we have
$$
F(0,0)=0, \qquad \pt{\tau}F(0,0)=1\neq0.
$$
Then the implicit function theorem implies that, in a neighborhood
of $(0,0)$, $\tau$ can be expressed as a function of $\nu$, that is,
there exists a positive constant $\nu_0$ such that
$$
\tau=\tau(\nu) \qquad\,\, \mbox{for}\,\, \nu\in[0,\nu_0].
$$
Furthermore, there exist positive constants $\alpha_1$ and
$\alpha_2$ such that, for any $\nu\in[0,\nu_0]$, we have
\begin{equation}\label{BgS07}
\alpha_1\nu^{\frac1{\gamma-1}}\leq\tau(\nu)\leq
\alpha_2\nu^{\frac1{\gamma-1}}.
\end{equation}
By \eqref{BgS05}, we conclude
\begin{equation}\label{BgS08}
u=O(1)\nu^{\frac1{\gamma-1}}\sTo0,\quad
v=O(1)\nu^{\frac1{\gamma-1}}\sTo0,\quad  \rho=O(1)\qquad\,\, \text{
as }\nu\sTo0,
\end{equation}
where $O(1)$ depends only on $\gamma$ and $b$. Thus,
\begin{equation}\label{BgS09}
M^2=\frac{q^2}{\rho^{\gamma-1}}=O(1)\nu^{\frac2{\gamma-1}}\sTo0
\qquad \text{ as }\, \nu\sTo0,
\end{equation}
where $q=\sqrt{u^2+v^2}$ is the flow speed and $O(1)$ depends only
on $\gamma$ and $b$.

We now analyze the flow field between the self-similar shock-front
$\mS$ and the straight-sided cone. Let $\omega_0$ be the vertex
angle of the cone and $\kappa=\ctg\omega_0$. Since the equations and
the boundary conditions are invariant under the scaling $(x,y)\to
(\alpha x, \alpha y), \alpha\ne 0$, we seek self-similar solutions
$(u,v)=(u, v)(\sigma), \sigma=\displaystyle x/y$, as in \cite{CoF}.
Then the flow field $(u,v)$ between the shock-front $\mS$ and the
cone $y=\kappa x$ is determined by the following free boundary value
problem:
\begin{align}
&\lb{\begin{aligned}
&\pt{\sigma}v + \sigma \pt{\sigma}u =
0,\\
&\big(1-\frac{u^2}{c^2}\big)\pt{\sigma}u - \big(\frac{2uv}{c^2} +
\sigma (1 - \frac{v^2}{c^2})\big)\pt{\sigma}v + v=0,\end{aligned}}
&&\text{ for }\sigma \in(\tau, \kappa),\label{315}\\
&(u,v)=\pr{u_S,v_S},&&\text{ on }\sigma=\tau,\label{316}\\
&u-\kappa v=0,&&\text{ on }\sigma=\kappa,\label{317}
\end{align}
where $\omega_0$ or $\kappa$ is unknown and determined together with
the solution, $\tau$ and $(u_S,v_S;\rho_S)$ are determined by the
shock polar and the flow direction $b$ right behind the shock-front
$\mS$ which are given in \eqref{BgS05}, and the density $\rho$ is
determined by Bernoulli's law with Bernoulli constant \eqref{BLT}.

\begin{figure}[htbp] \begin{center}
\includegraphics[height=2.2in,width=3in]{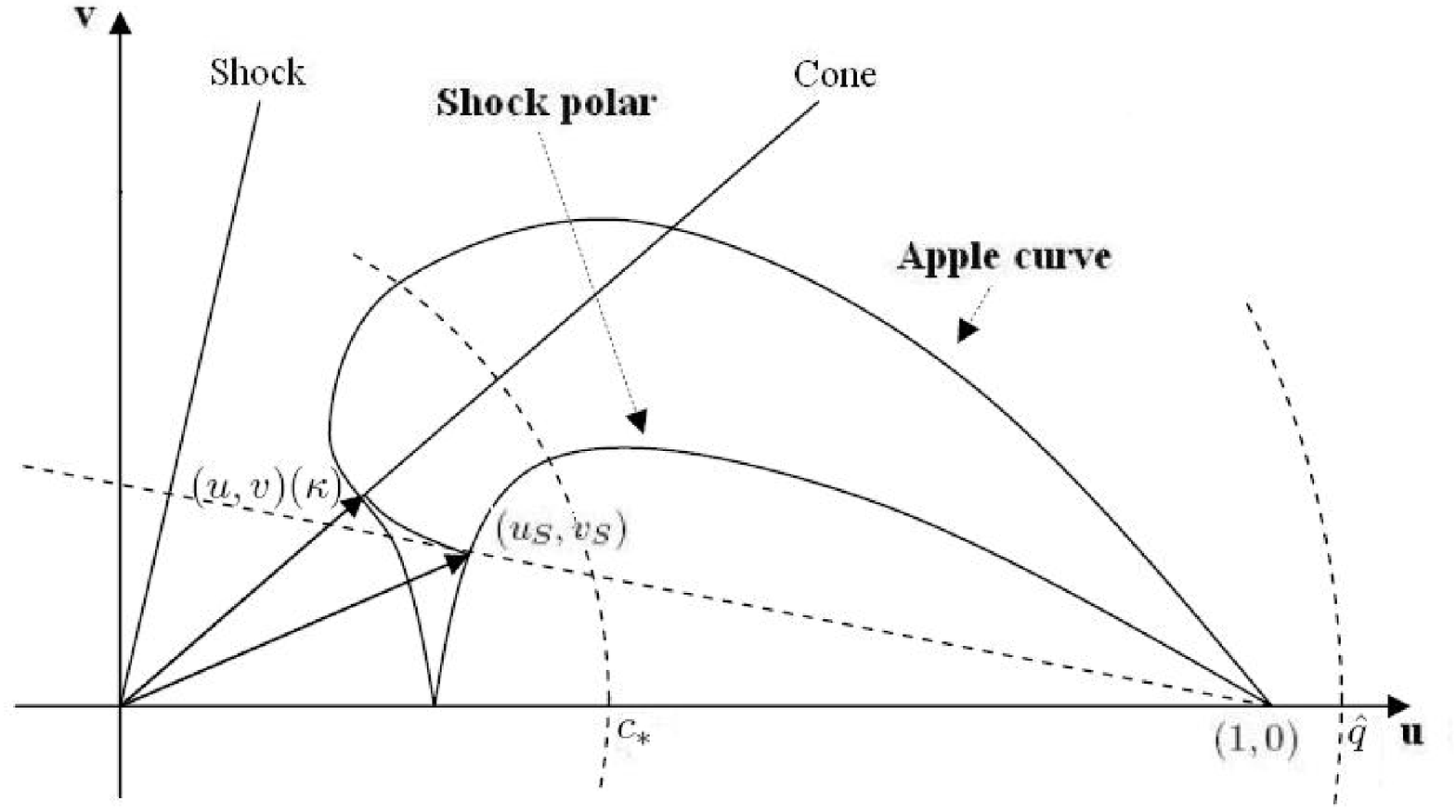}
\caption{Apple curve and shock polar for the self-similar solutions
} \label{fig:2}
\end{center}
\end{figure}

By \cite{CoF}, there exists a vertex angle $\omega_0=\omega_0(b)$ of
the cone and the corresponding self-similar solution $(u_0,
v_0)(\sigma),\ \sigma\in[\tau,k]$, between the shock-front and the
cone as the solution of the free boundary value problem
\eqref{315}--\eqref{317}. We assume that the flow between the
shock-front and the cone is subsonic, which is the case when
$M_\infty$ is large (equivalently, $\rho_\infty$ is small).
In this
case, we employ \eqref{315} to obtain
\[\begin{split}
&\Big(\big(1-\frac{u^2}{c^2}\big) + \frac{2uv}{c^2}\sigma +
\big(1-\frac{v^2}{c^2})\sigma^2\Big)\pt{\sigma}u + v = 0,\\
&\Big(\big(1-\frac{u^2}{c^2}\big)+ \frac{2uv}{c^2}\sigma +
\big(1-\frac{v^2}{c^2})\sigma^2\Big)\pt{\sigma}v - \sigma v = 0,\\
&\Big(\big(1-\frac{u^2}{c^2}) + \frac{2uv}{c^2}\sigma +
\big(1-\frac{v^2}{c^2}\big)\sigma^2\Big)\pt{\sigma}\big(\frac{q^2}{2}\big)
+ v(u-\sigma v) = 0,
\end{split}\]
where $q=\sqrt{u^2+v^2}$ is the flow speed. It is easy to verify
that
$$
u_0(\sigma)>0, \qquad v_0(\sigma)>0,
$$
and $u_0(\sigma)$, $q_0(\sigma)$, and the Mach number $M_0(\sigma)$
are strictly decreasing, while $v_0(\sigma)$ is strictly increasing,
with respect to $\sigma$. Therefore, we have
\begin{eqnarray*}
&&b=\frac{v(\tau)}{u(\tau)}< \frac{v(\kappa)}{u(\kappa)}=
\frac{1}{\kappa}=\tg\omega_0,\quad \text{ i.e.},\,\, \,\,
0<\kappa<\frac{1}{b},
\label{backg-1}\\
&&\max\limits_{\sigma\in[\tau,\kappa]}u_0(\sigma)= u_0(\tau),\label{backg-2}\\
&&\max\limits_{\sigma\in[\tau,\kappa]}v_0(\sigma)=
v_0(\kappa)< u_0(\tau)\tg\omega_0,\label{backg-3}\\
&&\max\limits_{\sigma\in[\tau,\kappa]}q_0(\sigma)\leq
q_0(\tau),\label{backg-4}\\
&&\max\limits_{\sigma\in[\tau,\kappa]}M_0(\sigma)\leq M_0(\tau)<1.
\label{backg-5}
\end{eqnarray*}

In the next sections, we develop a nonlinear iteration scheme and
establish the stability of self-similar transonic shocks under
perturbation of the upstream supersonic flow and the boundary
surface of the straight-sided cone.

\section{Stability Problem and Main Theorem}

In this section we first formulate the stability problem as a free
boundary value problem, then introduce a coordinate transformation
to reduce the free boundary problem into a fixed boundary value
problem, and finally state the main theorem (Theorem 3.1) of this
paper and its equivalent theorem (Theorem 3.2).

\subsection{Formulation of the stability problem}

The stability problem can be formulated as the following free
boundary problem.

\medskip
{\bf Problem I: Free boundary problem}. Determine the free boundary
$S=\{x=\phi(y)\}$ and the velocity field $(u,v)$ in the unbounded
domain $\{\phi(y)<x<\varphi^{-1}(y)\}$
satisfying the equations:
\begin{equation}
\label{202}
\lb{\begin{aligned}
&\pt{x}(\rho u) + \pt{y}(\rho v) +
\frac{\rho v}{y} = 0,\\
&\pt{x}v - \pt{y}u = 0,
\end{aligned}} \qquad\qquad \text{
in } \{\phi(y)<x<\varphi^{-1}(y)\},
\end{equation}
the free boundary conditions on $S$:
\begin{eqnarray}
&\label{203-1}\quad[\rho u][u] + [\rho v][v] = 0,\\
&\label{203-2}\quad-[v] = [u] \phi'(y),
\end{eqnarray}
and the slip boundary condition on the boundary surface of the
perturbed cone, $B=\{y=\varphi(x)\}$:
\begin{equation}\label{204}
v-\varphi'(x)\, u = 0 \qquad\quad \text{ on } B,
\end{equation}
where the density $\rho$ can be expressed as  a function of the
velocity $(u,v)$ by Bernoulli's law:
\begin{equation}\label{eq:Density}
\rho=\rho(q)=\Big(\tilde{\kappa}_\infty-\frac{\gamma-1}2q^2\Big)^{\frac1{\gamma-1}},
\end{equation}
with $q=\sqrt{u^2+v^2}$ and
$\tilde{\kappa}_\infty=(\gamma-1)\kappa_\infty$.

\medskip The equations in \eqref{202} can be rewritten in the
matrix form:
\begin{equation}\label{205}
A(U)\pt{x}U + B(U)\pt{y}U + C(y)U = 0,
\end{equation}
where $U=(u,v)^\top$ and
\[
A(U) =
\begin{pmatrix}1-\displaystyle\frac{u^2}{c^2}&-\displaystyle\frac{uv}{c^2}\\0&1\end{pmatrix},
\quad
B(U) =
\begin{pmatrix}-\displaystyle\frac{uv}{c^2}&1-\displaystyle\frac{v^2}{c^2}\\-1&0\end{pmatrix},
\quad
C(y)=\begin{pmatrix}0&\displaystyle\frac{1}{y}\\0&0\end{pmatrix}.
\]

To solve the free boundary problem ({\bf Problem I}), we introduce
the following coordinate transformation:
$$
\Pi_{\phi,\varphi}:\ (x,y)\mapsto(\xi,\eta)
$$
to fix the free boundary:
\begin{equation}\label{206}
\Pi_{\phi,\varphi}: \ \lb{\begin{aligned} &\xi - \eta\ctg\omega_1 =
x -
\phi(y),\\
&\eta - \xi\tg\omega_0 = y - \varphi(x).
\end{aligned}}
\end{equation}
Then the free boundary $S$ becomes a fixed boundary $\Gamma_1 =
\set{\xi=\eta\ctg\omega_1}$, and the domain
$\{\phi(y)<x<\varphi^{-1}(y)\}$ becomes a fixed domain
$$
\Omega = \set{\eta\ctg\omega_1<\xi<\eta\ctg\omega_0} =
\set{(r,\theta):\ \omega_0<\theta<\omega_1}.
$$
In transformation
\eqref{206}, $\phi$ as a function of $y$ is unknown and can be also
considered as a function of $\eta$ in the following way:
$$
\psi(\eta)\defs \phi(y(\eta\ctg\omega_1,\eta)).
$$

Then the transformation is written as
\begin{equation}\label{eq:FixBT}
\Pi_{\psi,\varphi}:\ \lb{\begin{aligned} &\xi - \eta\ctg\omega_1 = x
- \psi(\eta),
\\ &\eta - \xi\tg\omega_0 = y - \varphi(x).\end{aligned}}
\end{equation}
In the case that $\psi(\eta)$ is known, we can obtain the expression
of $\phi(y)$ from \eqref{eq:FixBT}. In fact, substituting
$\xi=\eta\ctg\omega_1$ into \eqref{eq:FixBT}, we have
\[
x=\psi(\eta),\qquad y=(1-\tg\omega_0\ctg\omega_1)\eta +
\varphi\circ\psi(\eta).
\]
Thus,
\[
\frac{\dif y}{\dif\eta}=1-\tg\omega_0\ctg\omega_1 +
\varphi'\,\dot{\psi},
\]
where $\varphi'=\displaystyle\frac{\dif\varphi(x)}{\dif x}$ and
$\dot\psi=\displaystyle\frac{\dif\psi(\eta)}{\dif\eta}$. In our
case, $\varphi'$ and $\dot\psi$ should be small perturbations to
$\tg\omega_0$ and $\ctg\omega_1$, respectively. Hence, we have
$\displaystyle\frac{\dif y}{\dif\eta}>0$, and $\eta$ can be also
expressed as a function of $y$, i.e.  $\eta=\eta(y)$. Then
$\phi(y)=\psi(\eta(y))$ is what we need. Therefore, we consider the
transformation with formulation \eqref{eq:FixBT} from now on. Then
we have
\begin{equation}\label{eq:FixBT-1}
y-\eta = \varphi(x) - \xi\tg\omega_0 =
\big(\varphi(x)-x\tg\omega_0\big) +
\tg\omega_0\big(\psi(\eta)-\eta\ctg\omega_1\big).
\end{equation}

A direct calculation indicates that the Jacobian matrix of the
transformation is
\begin{equation}\label{Eq:Jacobi1}
\frac{\partial(\xi,\eta)}{\partial(x,y)} =\frac{1}{1 -
\tg\omega_0(\ctg\omega_1-\dot\psi)}
\begin{pmatrix}1-\varphi'(x)(\ctg\omega_1 - \dot\psi)&\ctg\omega_1-
\dot\psi\\
\tg\omega_0-\varphi'(x) &1\end{pmatrix},
\end{equation}
or
\begin{equation}\label{Eq:Jacobi2}
\frac{\partial(x,y)}{\partial(\xi,\eta)} =
\begin{pmatrix}1&\dot\psi- \ctg\omega_1\\
\varphi'(x)-\tg\omega_0
&1+\varphi'(x)(\dot\psi-\ctg\omega_1)\end{pmatrix}.
\end{equation}
Then, under the transformation, system \eqref{205} becomes
\begin{equation}\label{208}
A(U)\pt{\xi}U + B(U)\pt{\eta}U + C(\eta)U = \ff(U;\psi) \qquad\text{
in }\Omega,
\end{equation}
where
$$
C(\eta)=\begin{pmatrix}0&\displaystyle {1}/{\eta}\\0&0\end{pmatrix},
$$
and
\[
\ff(U;\psi) = \tilde C(\eta;\psi)U - \tilde A(U;\psi)D_2U - \tilde
B(U;\psi)D_1U,
\]
and
\begin{eqnarray*}
&&\tilde A(U;\psi):=\frac{\tg\omega_0 - \varphi'}{1 -
\tg\omega_0(\ctg\omega_1-\dot\psi)}A(U),\\
&&\tilde B(U;\psi):=\frac{\ctg\omega_1 - \dot\psi}{1 -
\tg\omega_0(\ctg\omega_1-\dot\psi)}B(U),\\
&&\tilde C(\eta;\psi):=\begin{pmatrix}0&\displaystyle\frac{1}{\eta}-\frac{1}{y(\eta;\psi)}\\
0&0\end{pmatrix},\\
&&(D_1, D_2) :=(\pt{\xi} + \tg\omega_0\,\pt{\eta},
(\ctg\omega_1-\dot\psi)\,\pt{\xi} + \pt{\eta}).
\end{eqnarray*}
Since $\psi(\eta)=\phi(y(\eta\ctg\omega_1,\eta))$, we have
\[
\dot\psi=(\pt{\xi}y\ctg\omega_1+\pt{\eta}y)\phi' =
(1-\tg\omega_0\ctg\omega_1 + \varphi'\,\dot\psi)\phi',
\]
and the boundary condition \eqref{203-2} becomes
\begin{equation}\label{eq:RHC-B}
\dot\psi=\frac{1-\tg\omega_0\ctg\omega_1}{1-\varphi'(x)\phi'(y)}\phi'(y)
= -\frac{[v](1-\tg\omega_0\ctg\omega_1)}{[u]+\varphi'(x)\,[v]}.
\end{equation}

With these, the free boundary problem \eqref{202}---\eqref{204}
becomes the following fixed boundary problem:

\medskip
{\bf Problem II: Fixed boundary problem}. Determine the functions
$(U; \psi)=(u,v; \psi)$ in the unbounded domain:
$$
\Omega:=\set{\eta\ctg\omega_1<\xi<\eta\ctg\omega_0} =
\set{(r,\theta):\ \omega_0<\theta<\omega_1}
$$
satisfying system \eqref{208} and the boundary conditions:
\eqref{203-1} and \eqref{eq:RHC-B} on $\Gamma_1:=\{\xi=\eta \ctg
\omega_1\}$  and \eqref{204} on $\Gamma_0:=\{\xi=\eta\ctg
\omega_0\}$.

\subsection{Weighted spaces for solutions}

Based on the analysis of the self-similar transonic shock solutions
in Section 2 and the behavior of solutions to elliptic equations at
infinity, it is anticipated that the solutions have singularity at
the origin and decay at infinity. Thus, we need the following
weighted spaces as posed spaces to accommodate the features of
solutions to our problem.

\medskip
Let $1<q<\infty$ and $0\le\omega_0<\omega_1\le 2\pi$. Let
\[
\D:=\set{x\in\Real^{2}:\ 0<r<\infty, \omega_0<\theta<\omega_1}
\]
be an unbounded sector, where $(r,\theta)$ are the polar
coordinates. Then the boundary of the domain $\D$ consists of two
rays:
$$
\Gamma_{0}:=\set{x\in\Real^{2}:\ \theta=\omega_0, 0<r<\infty},
\qquad \Gamma_{1}=\set{x\in\Real^{2}:\ \theta=\omega_1, 0<r<\infty}.
$$
For any $k\in\Real$, $m=0,1,\cdots$,  we define the following
weighted Sobolev spaces $W^{m,q}_{(k)}(\D)$  as subspaces of $u\in
W^{m,q}_{loc}(\D)$:
\[
W^{m,q}_{(k)}(\D)=\set{u\in W_{loc}^{m,q}(\D):\
\norm{u}_{W_{(k)}^{m,q}(\D)}<\infty},
\]
with the norms:
\begin{equation}\label{eq:WeightedSobolevNorms}
    \norm{u(r,\theta)}_{W_{(k)}^{m,q}(\D)}
    =\norm{\me^{kt}u(\me^t,\theta)}_{W^{m,q}(\Pe(\D))},
\end{equation}
where
\begin{equation}\label{3.2a}
\Pe(r,\theta)=(t, \theta):=(\ln r, \theta)
\end{equation}
is a coordinate transformation from $(r, \theta)$ to $(t, \theta)$.

Define the norms for the trace of $u$ on each ray $\Gamma_{j}$ of
the boundary of $\D$ by
\begin{equation}\label{eq:WeightedTraceNorm}
    \norm{u(r,\omega_j)}_{W^{m -1/q,q}_{(k)}(\Gamma_{j})}=
    \norm{\me^{tk}u\pr{t,\omega_j}}_{W^{m-1/q,q}(\Real)}, \qquad j=0,1.
\end{equation}
It is easy to see that there exists a constant $K$, independent of
$u$, such that
\[
\norm{u}_{W^{m -{1}/{q},q}_{(k)}(\Gamma_{j})} \leq
K\,\norm{u}_{W^{m,q}_{(k)}(\D)}.
\]

Define
\begin{equation}\label{eq:WeightedHolderNorm}
    \norm{u(r,\theta)}_{C^{m}_{(k)}(\D)}
    =\norm{\me^{kt}u(\me^t,\theta)}_{C^{m}(\Pe(\D))},
\end{equation}
and denote by $C^{m}_{(k)}(\D)$ the space of functions with norm
$\norm{\cdot}_{C^{m}_{(k)}(\D)}$.

When $q>2$ and $m\geq1$, the well-known Sobolev imbedding theorem
implies that $W^{m,q}_{(k)}(\D)$ is embedded in $C^{m-1}_{(k)}(\D)$,
i.e., there exists a constant $K$, independent of $u$, such that
\begin{equation}\label{eq:EnbeddedInequality}
\norm{u}_{C^{m-1}_{(k)}(\D)}\leq K\,\norm{u}_{W^{m,q}_{(k)}(\D)}.
\end{equation}

For functions of single variable defined in $\Real_+$, we can also
define the following similar weighted norms:
\begin{equation}\label{eq:WeightedNormsOne}
\norm{u(r)}_{W_{(k)}^{m,q}(\Real_+)}=\norm{\me^{kt}u(\me^t)}_{W^{m,q}(\Real)},\,\,\,
\norm{u(r)}_{C^{m}_{(k)}(\Real_+)}=\norm{\me^{kt}u(\me^t)}_{C^{m}(\Real)}.
\end{equation}

\subsection{Main Theorem}

The main theorem of this paper is the following.

\begin{theorem}[Main theorem]\label{th:FixB}
Let $q>2$, $1<\gamma\leq2$, and $b>0$. Let
$$
\{(1,0;\rho_\infty);(u_0,v_0)(\sigma); \phi_1=y\ctg\omega_1\},
$$
with $\sigma\in [\ctg\omega_1,\ctg\omega_0]$ and
$b=\frac{v_0(\ctg\omega_1)}{u_0(\ctg\omega_1)}$,
form a transonic shock solution to \eqref{202} when the upstream
flow $(1,0;\rho_\infty)$ past the straight-sided cone with
$y=\varphi_0(x)=x\tg\omega_0$.
Then there exist positive constants $\nu_0$, $\varepsilon_0$, $M$,
and $M_S$ ($M$ and $M_S$ are independent of $\nu_0$ and
$\varepsilon_0$) such that, if the Mach number $M_\infty$ is
sufficiently large so that
$\nu:=\rho_\infty^{\gamma-1}=1/M_\infty^2\leq\nu_0$, then, for any
$0<\varepsilon\leq\varepsilon_0$ and
$\varepsilon\ll\nu^{\frac1{\gamma-1}}$, there exists a unique
solution $(U(\xi,\eta);\psi(\eta))$ to the fixed boundary value
problem \eqref{208}, \eqref{203-1}, \eqref{eq:RHC-B}, and
\eqref{204} satisfying $\psi(0)=0$ and the following estimates:
\begin{align}
&\norm{U-U_0}_{W^{1,q}_{(0)}(\Omega)}\leq
M\varepsilon,\label{eq:PerturbDownstream}\\
&\|\dot\psi-\ctg\omega_1\|_{\Gamma_1}\leq
M_S\varepsilon,\label{eq:PerturbShock}
\end{align}
with
$\norm{\cdot}_{\Gamma_1}\defs\norm{\cdot}_{W_{(0)}^{0,q}(\Gamma_1)}
+ \norm{\cdot}_{C^{0}(\Gamma_1)}$, provided that, if the perturbed
boundary $y=\varphi(x)$ of the cone satisfies $\varphi(0)=0$ and
\begin{equation}\label{eq:PerturbCone}
\norm{\varphi'(x)-\tg\omega_0}_{C^2_{(0)}(\Real_+)} +
\norm{\varphi'(x)-\tg\omega_0}_{W^{1,q}_{(0)}(\Real_+)}\leq\varepsilon,
\end{equation}
and the perturbed upstream flow field $U^-$ satisfies
\begin{equation}\label{eq:PerturbUpstream}
\norm{U^-}_{W^{1,q}_{(0)}(\Omega_e)} +
\norm{\pt{\xi}U^-}_{C^1_{(1)}(\Omega_e)}\leq\varepsilon,
\end{equation}
where $\Omega_e:=\{\eta \cot(\omega_1+\hat{\delta}_0)<\xi<\eta
\cot(\omega_0-\hat{\delta}_0)\}$ for some small $\hat{\delta}_0>0$,
\end{theorem}

Since $\Pi_{\psi,\varphi}$ or $\Pi_{\phi,\varphi}$ is invertible, we
conclude the following equivalent result from Theorem \ref{th:FixB}.

\begin{theorem}\label{th:Stability}
Suppose that the assumptions of Theorem {\rm \ref{th:FixB}} hold.
Then there exist positive constants $\nu_0$, $\varepsilon_0$, $M$,
and $M_S$ ($M$ and $M_S$ are independent of $\nu_0$ and
$\varepsilon_0$) such that, if the Mach number $M_\infty$
is sufficiently large so that
$\nu:=\rho_\infty^{\gamma-1}=1/M_\infty^2\leq\nu_0$, then, for any
$0<\varepsilon\leq\varepsilon_0$ and
$\varepsilon\ll\nu^{\frac1{\gamma-1}}$,  there exists a unique
solution (still denoted by) $(U(x,y);\phi(y))$ to the free boundary
problem \eqref{202}--\eqref{204}, provided that, if the boundary
surface $y=\varphi(x)$ of the perturbed cone satisfies
$\varphi(0)=0$ and
\begin{equation}\label{eq:PerturbCone-a}
\norm{\varphi'(x)-\tg\omega_0}_{C^2_{(0)}(\Real_+)} +
\norm{\varphi'(x)-\tg\omega_0}_{W^{1,q}_{(0)}(\Real_+)}\leq\varepsilon,
\end{equation}
and the perturbed upstream flow field $U^-$ satisfies
\begin{equation}\label{eq:PerturbUpstream-a}
\norm{U^-}_{W^{1,q}_{(0)}(\Omega_e)} + \norm{\nabla
U^-}_{C^1_{(1)}(\Omega_e)}\leq\varepsilon,
\end{equation}
where $\Omega_e:=\{y\cot(\omega_1+\delta_0)<x<
y\cot(\omega_0-\delta_0)\}$ for some small $\delta_0>0$.
Moreover, the solution $(U(x,y); \phi(y))$ satisfies
$\phi(0)=0$ and the following estimates:
\begin{align}
&\label{eq:PerturbDS}\norm{ U\circ\Pi_{\phi,\varphi}^{-1}
-U_0\circ\Pi_{\phi_0,\varphi_0}^{-1}}_{W^{1,q}_{(0)}(\Omega)}\leq
M\varepsilon,\\
&\label{eq:PerturbS}\norm{\phi'-\ctg\omega_1}_{S}\leq
M_S\varepsilon,
\end{align}
where $\phi'=\frac{\dif\phi}{\dif y}$ and
$\norm{\cdot}_{S}:=\norm{\cdot}_{W_{(0)}^{0,q}(\Real_+)} +
\norm{\cdot}_{C^{0}(\Real_+)}$.
\end{theorem}

\begin{figure}[htbp] \begin{center}
\includegraphics[height=2.5in,width=3.5in]{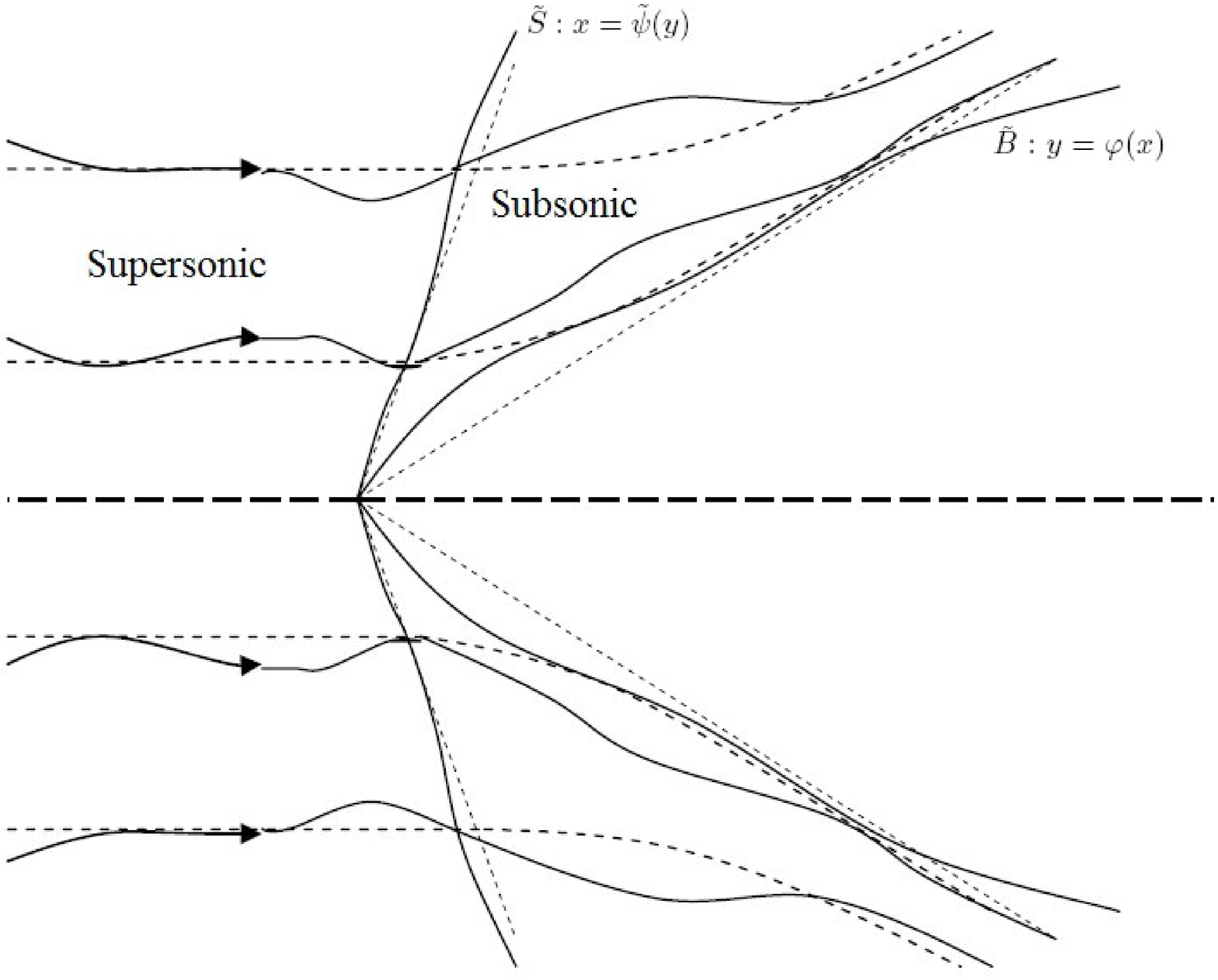}
\caption{Stability of transonic shock-front solutions} \label{fig:3}
\end{center}
\end{figure}

\begin{rem}
The existence of the perturbed upstream flow field $U^-$ satisfying
\eqref{eq:PerturbUpstream-a} can be obtained by blowing up the
angular point and then following the standard argument as in Li-Yu
\cite{LiYu}, since the equations are still quasilinear hyperbolic
under the transformation.
%
\end{rem}

\begin{rem}
Estimates \eqref{eq:PerturbDS} and \eqref{eq:PerturbS} imply that
the downstream flow and the transonic shock-front are a perturbation
of the self-similar transonic shock solution. Hence, the
self-similar transonic shock-front is conditionally stable with
respect to the conical perturbation of the boundary surface of the
cone and the upstream flow in the function spaces with restrictions
on the downstream flow field both at the corner and at infinity.
\end{rem}

\section{Well-posedness for a Singular Linear Elliptic Problem}

In this section, we establish the well-posedness for a singular
linear elliptic equation, which will play an essential role for
establishing the main theorem, Theorem \ref{th:FixB}.

Let $0<\omega_0<\omega_1<\frac{\pi}{2}$ and set
\[
\begin{split}
&\Omega:= \set{(x,y)\in\Real^2:\ 0<r<\infty,\omega_0<\theta<\omega_1},\\
&\Gamma_0:= \set{(x,y)\in\Real^2:\ 0<r<\infty,\theta=\omega_0},\\
&\Gamma_1:= \set{(x,y)\in\Real^2:\ 0<r<\infty,\theta=\omega_1},
\end{split}
\]
where $(r,\theta)$ are the polar coordinates in the plane.

\subsection{Neumann problem for a singular second-order elliptic equation}

Consider the following Neumann boundary value problem in $\Omega$:
\begin{equation}\label{001}
\lb{
\begin{aligned}
&L_0\varphi\defs\pt{xx}\varphi + \pt{yy}\varphi +
\frac{\pt{y}\varphi}{y} = f \quad & \text{ in }\Omega,\\
&B_0\varphi\defs\pt{y}\varphi - \tg\omega_0\pt{x}\varphi = g_0 &
\text{ on }\Gamma_0,\\
&B_1\varphi\defs\pt{x}\varphi - \ctg\omega_1\pt{y}\varphi = g_1 &
\text{ on }\Gamma_1.
\end{aligned} }
\end{equation}
We have the following proposition.

\begin{proposition}\label{le:001}
Let $1<q<\infty$. The operator $(L_0,B_0,B_1)$ defined in
\eqref{001} realizes an isomorphism from $W_{(-1)}^{2,q}(\Omega)$ to
$W_{(1)}^{0,q}(\Omega)\times
(W_{(0)}^{1-1/q,q}(\Real_{+}))^2$.
Moreover, we have the following
estimate for the solution to problem \eqref{001}:
\begin{equation}\label{002}
\norm{\varphi}_{W_{(-1)}^{2,q}(\Omega)} \leq
K\,\Big(\norm{f}_{W_{(1)}^{0,q}(\Omega)} +
\sum_{j=0,1}\norm{g_j}_{W_{(0)}^{1-1/q,q}(\Real_{+})}\Big),
\end{equation}
where the constant $K$  is independent of $\varphi$, but depends
only on $q$ and $\omega_0$ (actually $\ctg\omega_0$).
\end{proposition}

To prove this proposition, we employ a criterion identified by
Hartman-Wintner \cite{HW} for the uniqueness of solutions to the
Dirichlet boundary value problem for systems of second-order
differential equations. For self-containedness, we give a brief
description here; for more details, see \cite{HW}.

\begin{lemma}\label{le:ODEUniqueness}
Consider the following boundary value problem for the system of
second-order differential equations for $x\in\Real^n$:
\begin{equation}\label{eq:ODEsBVP}
\lb{\begin{aligned}&x'' + A_1(t)\, x' + A_2(t)\, x =
0\qquad\text{for}\,\,
t\in(t_0,t_1),\\
&x(t_0)=x(t_1)=0,
\end{aligned}}
\end{equation}
where $A_1(t)$ and $A_2(t)$ are $n\times n$ real matrices. Assume
that there exists a matrix $K(t)$ such that
\begin{equation}\label{eq:Criterion}
N=(K^0)'-A_2^0-\big(\frac12A_1 - K^0\big)\big(\frac12A_1^\top -
K^0)> 0,
\end{equation}
where $K^0=\frac12(K+K^\top)$ and $A_2^0=\frac12(A_2+A_2^\top)$.
Then problem \eqref{eq:ODEsBVP} has only the trivial solution
$x\equiv0$.
\end{lemma}

\begin{proof}
Taking the inner product on the equations with $x$ and integrating
from $t_0$ to $t_1$ yields
\begin{equation}\label{eq:Criterion01}
\int_{t_0}^{t_1}\pr{x'\cdot x' - x\cdot A_1x' -x\cdot A_2x}\dif t=0.
\end{equation}
Since $(x\cdot Kx)'=2x'\cdot K^0x+x\cdot K'x$, we have
\begin{equation}\label{eq:Criterion02}
0=\int_{t_0}^{t_1}(x\cdot Kx)'\dif t=\int_{t_0}^{t_1}(2x'\cdot K^0x
+ x\cdot K'x)\, \dif t.
\end{equation}
Then
\[
\int_{t_0}^{t_1}\Big(\big|x'-\big(\frac12A_1^\top-K^0\big)x\big|^2 +
x\cdot L x\Big)\dif t = 0,
\]
where $L:= K'-A_2-\big(\frac12A_1 - K^0\big)\big(\frac12A_1^\top -
K^0)$.

Similarly, we have
\[
\int_{t_0}^{t_1}\Big(\big|x'-\big(\frac12A_1^\top-K^0\big)x\big|^2 +
x\cdot L^\top x\Big)\dif t = 0.
\]
Combining the above two identities, we obtain
\[
\int_{t_0}^{t_1}\Big(\big|x'-\big(\frac12A_1^\top-K^0\big)x\big|^2 +
x\cdot N x\Big)\dif t = 0.
\]
Since $N$ is positive definite, we conclude $x\equiv0$.
\end{proof}

\begin{proof}[Proof of Proposition {\rm \ref{le:001}}]
Rewriting the boundary value problem \eqref{001} in the polar
coordinates $(r,\theta)$, we have
\begin{equation}\label{0021}
\lb{
\begin{aligned}
&L_0\varphi=(r\pt{r})^2\varphi + \pt{\theta}^2\varphi +
r\pt{r}\varphi + \ctg\theta\, \pt{\theta}\varphi = r^2f  &\text{in }\Omega,\\
&B_0\varphi=\pt{\theta}\varphi = rg_0 & \text{ on }\Gamma_0,\\
&B_1\varphi=\pt{\theta}\varphi = rg_1 & \text{ on }\Gamma_1.
\end{aligned}
}
\end{equation}

Employing the transformation $\Pe$ in \eqref{3.2a}, i.e.,
$\Pe(r,\theta)=(t,\theta)=(\ln r, \theta)$, we convert the infinite
sector $\Omega$ into an infinite strip: $\D:=\set{(t,\theta):\
t\in\Real,\ \omega_0<\theta<\omega_1}$. Accordingly, the boundary
value problem \eqref{0021} is converted to the following boundary
value problem in $\D$:
\begin{equation}\label{0022}
\lb{
\begin{aligned}
&\pt{tt}\varphi + \pt{\theta\theta}\varphi +
\pt{t}\varphi + \ctg\theta\, \pt{\theta}\varphi = \me^{2t}f & \text{ in }\D,\\
&\pt{\theta}\varphi = \me^tg_0&
\text{ on }\Sigma_0,\\
&\pt{\theta}\varphi = \me^tg_1& \text{ on }\Sigma_1.
\end{aligned} }
\end{equation}

Applying the Fourier transformation $\ff_{t\sTo\lambda}$ with
respect to $t$, we obtain a family of boundary value problems with
complex parameter $\lambda$:
\begin{equation}\label{0023}
\lb{
\begin{aligned}
&\hat{\varphi}'' + \ctg\theta\,\hat{\varphi}' + (-\lambda^2
+ \mi\lambda)\varphi = \widehat{\me^{2t}f} \qquad & \theta\in(\omega_0,\omega_1),\\
&\hat{\varphi}' = \widehat{\me^tg_0}\qquad&
\theta=\omega_0,\\
&\hat{\varphi}' = \widehat{\me^tg_1} \qquad & \theta=\omega_1.
\end{aligned} }
\end{equation}

We now employ a Fredholm-type theorem, Theorem A.1 in Appendix,
to find that, if the homogeneous problems of \eqref{0023} (i.e.
$f=g_0=g_1=0$) have only the trivial solution $\hat{\varphi}\equiv
0$ for all $\lambda$ with $\text{Im} \lambda = -1$, then, for any
$(f,g_0,g_1)$ such that $\me^{t}f\in W^{0,q}(\D)$ and $g_j\in
W^{1-1/q,q}(\Sigma_j), j=0,1$, the boundary value problem
\eqref{0022} in the infinite strip $\D$ has a unique solution
$\varphi$ such that $\me^{-t}\varphi\in W^{2,q}(\D)$. Moreover, the
solution $\varphi$ satisfies the estimate:
\begin{equation}
\norm{\me^{-t}\varphi}_{W^{2,q}(\D)}\leq K\,
\Big(\norm{\me^{t}f}_{W^{0,q}(\D)} +
\sum_{j=0,1}\norm{g_j}_{W^{1-1/q,q}(\Sigma_j)}\Big),
\end{equation}
where $K$ is independent of $\varphi$, but does depend on $\omega_0$
because of the coefficient $\ctg\theta$. Then the results of
Proposition \ref{le:001} follow. Therefore, it suffices to verify
that, in the case that $f= g_1=g_2\equiv0$, the boundary value
problems \eqref{0023} with complex parameter $\lambda$, $\text{Im}
\lambda=-1$, have only the trivial solution $\hat{\varphi}\equiv0$.

When $f= g_0=g_1\equiv0$, we set $y=\hat{\varphi}'$ to have
\begin{equation}\label{0024}
\lb{
\begin{aligned}
&y'' + \ctg\theta\ y' + (-\csc^2\theta-\lambda^2
+ \mi\lambda)\,\varphi = 0 \qquad& \text{for }\, \theta\in(\omega_0,\omega_1),\\
&y = 0 &\text{for}\,\, \theta=\omega_0,\\
&y = 0& \text{for}\,\, \theta=\omega_1.
\end{aligned} }
\end{equation}

Write $y=y_1+ y_2 \mi$ and $\lambda=\mu-\mi$. Then
$-\lambda^2+\lambda \mi = -\mu^2 + 2 + 3\mu \mi$, and \eqref{0024}
can be rewritten as the following boundary value problems of
second-order differential equations with real coefficients:
\begin{equation}\label{0025}
\lb{
\begin{aligned}
&\begin{aligned} &\begin{pmatrix}y_1\\y_2\end{pmatrix}'' +
\begin{pmatrix}\ctg\theta&0\\0&\ctg\theta\end{pmatrix}
\begin{pmatrix}y_1\\y_2\end{pmatrix}' \\
&\quad+ \begin{pmatrix}-\csc^2\theta - \mu^2 +
2&-3\mu\\3\mu&-\csc^2\theta - \mu^2 + 2\end{pmatrix}
\begin{pmatrix}y_1\\y_2\end{pmatrix} = 0 \qquad
\,\,\text{for}\,\, \theta\in(\omega_0,\omega_1),
\end{aligned}
\\
&\begin{pmatrix}y_1\\y_2\end{pmatrix} =
0\qquad\quad\qquad\qquad\qquad\qquad\qquad\qquad\qquad\qquad\qquad
\quad\text{on}\,\,\theta=\omega_0, \omega_1.
\end{aligned}
}
\end{equation}

Let $K(\theta)=K^0(\theta)=\begin{pmatrix} \tg\theta&0\\0&
\tg\theta\end{pmatrix}$.
Then
$$
N=\begin{pmatrix} a(\theta,\mu)&0\\0& a(\theta,\mu)\end{pmatrix},
$$
where
\[
\begin{split}
a(\theta,\mu) :=&(\tg\theta)' - (- \csc^2\theta - \mu^2 + 2)
-(\frac{1}{2}\ctg\theta - \kappa)^2\\
=&\mu^2 + \csc^2\theta - \frac{1}{4}\ctg^2\theta>0,\\
\end{split}
\]
which implies that the symmetric matrix $N$ is positive definite and
hence satisfies the criterion \eqref{eq:Criterion}. By Lemma
\ref{le:ODEUniqueness}, we obtain that $y\equiv0$. That is,
$\hat{\varphi}\equiv const.$  Then the equations in \eqref{0023}
yields that $\hat{\varphi}\equiv 0$ for any $\text{Im}\lambda=-1$.
This completes the proof.
\end{proof}

Proposition \ref{le:001} can be directly applied to a special
boundary value problem of first-order partial differential
equations.

\subsection{A boundary value problem of a singular first-order elliptic system}

Consider the boundary value problem for the first-order system:
\begin{align}
&\hat A\,\pt{x}U + \hat B\,\pt{y}U + \hat C\, U = F &&\text{ in }\Omega,\label{eq:BasicEqs}\\
&U\cdot\hat\alpha_0=g_0 &&\text{ on }\Gamma_0,\label{eq:BasicBC1}\\
&U\cdot\hat\alpha_1=g_1 &&\text{ on }\Gamma_1,\label{eq:BasicBC2}
\end{align}
where $U=(u,v)^\top$, $\hat\alpha_0=(-\tg\omega_0,1)^\top$,
$\hat\alpha_1=(1,-\ctg\omega_1)^\top$, $F=(f_1,f_2)^\top$, and
\begin{equation}\label{4.16}
\hat A=\begin{pmatrix}1&0\\0&1\end{pmatrix},\qquad \hat
B=\begin{pmatrix}0&1\\-1&0\end{pmatrix},\qquad \hat
C=\begin{pmatrix}0&\frac1y\\0&0
\end{pmatrix}.
\end{equation}

To solve this problem, we first construct a function $\Phi\in
W^{2,q}_{(-1)}(\Omega)$ such that
\begin{equation}\label{LT01}
\lb{\begin{aligned} &\triangle\Phi = f_2 &\text{ in }\Omega,\\ &\Phi
= 0&\text{ on }\Gamma_0,\\ &\Phi = 0 &\text{ on }\Gamma_1.
\end{aligned}}
\end{equation}
By virtue of \cite{MP}, there exists a unique $\Phi$ with the
following estimate:
\begin{equation}\label{LT02}
\norm{\Phi}_{W^{2,q}_{(-1)}(\Omega)}\leq
K\,\norm{f_2}_{W^{0,q}_{(1)}(\Omega)},
\end{equation}
where $K$ is independent of $\Phi$.

Let $\tilde u=u+\pt{y}\Phi$ and $\tilde v=v-\pt{x}\Phi$. Then the
boundary value problem \eqref{eq:BasicEqs}---\eqref{eq:BasicBC2} is
reduced to
\begin{align}
&\hat A\,\pt{x}\tilde U + \hat B\,\pt{y}\tilde U + \hat C\,\tilde U
=\tilde F&&\text{ in }\Omega,\label{LT03}\\
&\tilde U\cdot\hat\alpha_0=\tilde g_0&&\text{ on }\Gamma_0,\label{LT04}\\
&\tilde U\cdot\bar\alpha_1=\tilde g_1&&\text{ on
}\Gamma_1,\label{LT05}
\end{align}
where $\tilde F=(\tilde f,0)^\top$, $\tilde
f=f_1-\displaystyle\frac{\pt{x}\Phi}{y}$, $\tilde g_0=g_0$, and
$\tilde g_1=g_1$. Since
\[\begin{split}
\norm{\frac{\pt{x}\Phi}{y}}_{W_{(1)}^{0,q}(\Omega)} =&
\norm{\me^t\,\frac{\pt{x}\Phi(\me^t,\theta)}{\me^t\sin\theta}}_{W^{0,q}(\Pe(\Omega))}
\leq
K(\omega_0)\norm{\pt{x}\Phi(\me^t,\theta)}_{W^{0,q}(\Pe(\Omega))}\\
=&K(\omega_0)\norm{\pt{x}\Phi}_{W_{(0)}^{0,q}(\Omega)} \leq
K(\omega_0)\norm{f_2}_{W_{(1)}^{0,q}(\Omega)},
\end{split}
\]
we have
\[
\norm{\tilde f}_{W_{(1)}^{0,q}(\Omega)}\leq K(\omega_0)\sum_{j=1,2}
\norm{f_j}_{W_{(1)}^{0,q}(\Omega)}=C(\omega_0)
\norm{F}_{W_{(1)}^{0,q}(\Omega)}.
\]

By the second equation of \eqref{LT03}, there exists  a potential
function $\varphi$ such that
$\nabla\varphi=(\pt{x}\varphi,\pt{y}\varphi)=\tilde{U}$. Then the
boundary value problem \eqref{LT03}---\eqref{LT05} can be
reformulated as a boundary value problem of a second-order elliptic
equation:
\[
\lb{
\begin{aligned}
&\pt{xx}\varphi + \pt{yy}\varphi +
\frac{\pt{y}\varphi}{y} =\tilde f& \text{ in }\Omega,\\
&\pt{y}\varphi - \tg\omega_0\pt{x}\varphi =\tilde g_0&
\text{ on }\Gamma_0,\\
&\pt{x}\varphi - \ctg\omega_1\pt{y}\varphi =\tilde g_1& \text{ on
}\Gamma_1.
\end{aligned} }
\]
Now Proposition \ref{le:001} yields that there exists a unique
solution $\varphi\in W_{(-1)}^{2,q}(\Omega)$ with the following
estimate:
\begin{equation}\label{LT06}
\begin{split}
\norm{\varphi}_{W_{(-1)}^{2,q}(\Omega)} \leq&
\,K\,\Big(\|\tilde{f}\|_{W_{(1)}^{0,q}(\Omega)} +
\sum_{j=0,1}\norm{\tilde{g_j}}_{W_{(0)}^{1-1/q,q}(\Gamma_j)}\Big)\\
\leq&\, K\,\Big(\norm{F}_{W_{(1)}^{0,q}(\Omega)} +
\sum_{j=0,1}\norm{g_j}_{W_{(0)}^{1-1/q,q}(\Gamma_j)}\Big),
\end{split}
\end{equation}
where $K$ depends only on $\omega_0$, but is independent of
$\varphi$, $F$, and $g_j, j=0,1$. Thus, there exists a unique
solution $U\in (W_{(0)}^{1,q}(\Omega))^2$ to problem
\eqref{eq:BasicEqs}---\eqref{eq:BasicBC2} with the following
estimate:
\begin{equation}\label{eq:BasicEstimate}
\norm{U}_{W_{(0)}^{1,q}(\Omega)}
\leq \hat{K}\,\Big(\norm{F}_{W_{(1)}^{0,q}(\Omega)} +
\sum_{j=0,1}\norm{g_j}_{W_{(0)}^{1-1/q,q}(\Gamma_j)}\Big),
\end{equation}
where $\hat{K}$ is independent of $U$, $F$, and $g_j, j=0,1$, but
depends only on $\omega_0$.

With the argument above, we obtain the following corollary of
Proposition \ref{le:001}:

\begin{proposition}
Let $1<q<\infty$. Let $F\in (W_{(1)}^{0,q}(\Omega))^2$ and $g_j\in
W_{(0)}^{1-1/q,q}(\Gamma_j)$, $j=0,1$. Then there exists a unique
solution $U\in (W_{(0)}^{1,q}(\Omega))^2$ to the boundary value
problem \eqref{eq:BasicEqs}---\eqref{eq:BasicBC2}. Moreover, the
solution satisfies estimate \eqref{eq:BasicEstimate}.
\end{proposition}

Applying the continuity method, we can extend this result to a small
``perturbed'' boundary value problem for the first-order elliptic
system.

\subsection{A small perturbed boundary value problem for a
singular first-order elliptic system}

Consider the following boundary value problem:
\begin{align}
&A\,\pt{x}U + B\,\pt{y}U + C\,U = F&&\text{ in }\Omega,\label{eq:PerturbEqs}\\
&U\cdot\alpha_0=g_0 &&\text{ on }\Gamma_0,\label{eq:PerturbBC1}\\
&U\cdot\alpha_1=g_1 &&\text{ on }\Gamma_1,\label{eq:PerturbBC2}
\end{align}
where $A$, $B$, and $C$ are $2\times2$ matrix functions defined on
$\Omega$, $\alpha_j=(\alpha_{j1},\alpha_{j2})^\top$ are vector
functions defined on $\Gamma_j$, $j=0,1$. Then we have

\begin{proposition}\label{le:002}
There exists a positive constant $\hat{\epsilon}$, depending only on
the constant $\hat{K}$ on the right side of estimate
\eqref{eq:BasicEstimate}, such that, if the coefficients of problem
\eqref{eq:PerturbEqs}---\eqref{eq:PerturbBC2} satisfy the following
conditions:
\begin{equation}\label{eq:Perturb}
\norm{(A-\hat A,B-\hat{B})}_{C^0(\Omega)}+ \norm{C-\hat
C}_{C^0_{(1)}(\Omega)} +
\sum_{j=0,1}\norm{\alpha_j-\hat\alpha_j}_{C^1_{(0)}(\Gamma_j)}\leq
\hat{\epsilon},
\end{equation}
then, for any $F\in (W_{(1)}^{0,q}(\Omega))^2$ and $g_j\in
W_{(0)}^{1-1/q,q}(\Gamma_j), j=0,1$, there exists a unique solution
$U\in (W_{(0)}^{1,q}(\Omega))^2$ to the boundary value problem
\eqref{eq:PerturbEqs}---\eqref{eq:PerturbBC2}. Moreover, the
solution $U$ satisfies the following estimate:
\begin{equation}\label{eq:PerturbEstimate}
\norm{U}_{W_{(0)}^{1,q}(\Omega)}
\leq K\,\Big(\norm{F}_{W_{(1)}^{0,q}(\Omega)} +
\sum_{j=0,1}\norm{g_j}_{W_{(0)}^{1-1/q,q}(\Gamma_j)}\Big),
\end{equation}
where $K$ is independent of $U$, $F$, and $g_j, j=0,1$, but depends
only on $\omega_0$.
\end{proposition}

\begin{proof}
Denote the boundary value problems
\eqref{eq:BasicEqs}---\eqref{eq:BasicBC2} and
\eqref{eq:PerturbEqs}---\eqref{eq:PerturbBC2} by linear bounded
operators $\hat T$ and $T$ respectively from
$X=(W_{(0)}^{1,q}(\Omega))^2$ to
$Y=(W_{(1)}^{0,q}(\Omega))^2\times\prod\limits_{j=0,1}
W_{(0)}^{1-1/q,q}(\Gamma_j)$. By Proposition \ref{le:002}, $\hat T$
is invertible and $\hat T^{-1}$ is also a linear bounded operator.

Let $T_s=(1-s)\hat T + sT$, $s\in[0,1]$. By virtue of
\eqref{eq:Perturb}, we have
\[
\|(\hat T-T)U\|_{Y}\leq K\,\hat{\epsilon}\norm{U}_{X},
\]
where $K$ is a constant. Since $\hat TU=T_sU+s(\hat T-T)U$, by
Proposition \ref{le:002},
\[
\norm{U}_{X}\leq K\,\pr{\norm{T_sU}_Y + \|(\hat T-T)U\|_{Y}}\leq
K\,\norm{T_sU}_Y + K\,\hat{\epsilon}\norm{U}_{X}.
\]
Choosing $\hat{\epsilon}$ sufficiently small such that
$K\hat{\epsilon}<1$,  we have
\begin{equation}\label{LT07}
\norm{U}_{X}\leq K\,\norm{T_sU}_Y,
\end{equation}
where $K$ is independent of $U$ and $s\in[0,1]$, but depends only on
$\omega_0$.

Then, applying the continuity method, Proposition \ref{le:002}, and
the uniform estimates \eqref{LT07}, we completes the proof.
\end{proof}

\section{Iteration Scheme}

Our iteration scheme for the stability problem consists of two
iteration mappings: One is for an iteration of approximate transonic
shock-fronts; and the other is for an iteration of the corresponding
nonlinear boundary value problems for given approximate
shock-fronts.

Let $q>2$ and $\psi_0(\eta)=\eta\ctg\omega_1$. Define
\begin{equation}\label{301}
\begin{split}
\Sigma_\tau=&\set{\psi:\ \psi(0)=0, \|\dot\psi-\dot\psi_0\|_{\Gamma_1}\leq\tau},\\
O_\tau=&\Big\{\delta U=(\delta u,\delta v)^\top:\ \norm{(\delta
u,\delta v)}_{W_{(0)}^{1,q}(\Omega)}\leq\tau\Big\}.
\end{split}
\end{equation}

Let $M_S$ and $M$ are positive constants to be determined later. In
order to find the perturbed shock solution to the fixed boundary
value problem \eqref{208}, \eqref{203-1}, \eqref{eq:RHC-B}, and
\eqref{204} of the self-similar shock solution $(U_0;U_0^-;\psi_0)$,
our strategy is as follows: Let $\varepsilon_0>0$ be a small
constant to be determined later and
$0<\varepsilon\leq\varepsilon_0$. Given an approximate boundary
$\psi\in\Sigma_{M_S\varepsilon}$, solve the nonlinear boundary value
problem \eqref{208}, \eqref{203-1}, and \eqref{204} to obtain a
perturbed solution $U_\psi$ of $U_0$. Then we use one of the
Rankine-Hugoniot conditions, \eqref{eq:RHC-B}, to update the
approximate boundary and obtain new $\psi_*$:
\begin{equation}\label{eq:UpdateApS}
\lb{\begin{aligned}&\dot\psi_* =
-\frac{[v](1-\tg\omega_0\ctg\omega_1)} {[u]+\varphi'(x)[v]},\\
&\psi_*(0)=0.\end{aligned}}
\end{equation}
This defines an iteration mapping: $\fj_S:\ \psi\mapsto\psi_*$. To
prove Theorem \ref{th:FixB}, it suffices to verify that there exist
positive constants $M_S$ and $\varepsilon_0$ such that $\fj_S$ is a
well-defined, contraction mapping in $\Sigma_{M_S\varepsilon}$ for
any $0<\varepsilon\leq\varepsilon_0$.

Since the initial value problem \eqref{eq:UpdateApS} is easier, we
will focus mainly on the nonlinear boundary value problem
\eqref{208}, \eqref{203-1}, and \eqref{204} for given
$\psi\in\Sigma_{M_S\varepsilon}$, which requires another nonlinear
iteration:
For given $\delta U\in O_{M\varepsilon}$, a linearized boundary
value problem will be solved in
the weighted Sobolev space $W^{1,q}_{(0)}(\Omega)$ to obtain a
unique solution $\delta U_*$ that is defined as an iteration mapping
$\fj:\ \delta U\mapsto\delta U_*$. By showing that there exist
positive constants $M$ and $\varepsilon_0$ such that $\fj$ is a
well-defined contraction mapping in $O_{M\varepsilon}$ for any
$0<\varepsilon\leq\varepsilon_0$, we conclude that the nonlinear
problem \eqref{208}, \eqref{203-1}, and \eqref{204} is uniquely
solvable in the weighted Sobolev space $W^{1,q}_{(0)}(\Omega)$ as a
perturbation to the background self-similar transonic shock
solution.

\medskip
In particular,
the {\it linearized problem} to \eqref{208}, \eqref{203-1}, and
\eqref{204} in the iteration $\fj$ is
\begin{align}
\label{302}&A_0\,\pt{\xi}\delta U_* + B_0\,\pt{\eta}\delta U_* +
C(\eta)\,\delta U_* = F(\delta U;\psi)
&\text{ in }\Omega,\\
\label{303}&\delta v_* - \varphi'(x(\xi,\xi\tg\omega_0))\,\delta u_*
=g_0(\delta U;\psi) &\text{ on }\Gamma_0,\\
\label{304}&\alpha\,\delta u_* + \beta\,\delta v_* = g_1(\delta U;
\psi) &\text{ on }\Gamma_1,
\end{align}
where $A_0=A(U_0)$ and $B_0=B(U_0)$ for
 the background solution  $U_0=U_0(\theta)=(u_0,v_0)(\theta)$ between
$\Gamma_0$ and $\Gamma_1$ described in Section 2, and
\begin{equation}\label{201}
\alpha=\displaystyle\frac{\partial G}{\partial
u}(U_0(\omega_1);1,0),\qquad \beta=\displaystyle\frac{\partial
G}{\partial v} (U_0(\omega_1); 1,0),
\end{equation}
for
\begin{equation}\label{201a}
G\pr{U; U^-}\defs[\rho u][u] + [\rho v][v].
\end{equation}
We denote this linearized problem as a linear operator $\ft:\ \delta
U\mapsto(F; g_0,g_1)$ for $F=(f_1,f_2)$.

Since $(\pt{u},\pt{v})\rho= -\rho^{2-\gamma}(u,v)$,
we have
\[
\begin{split}
\pt{u}G\pr{U;U^-}&= [\rho u] + (\rho + u\pt{u}\rho)[u] +
v\pt{v}\rho[v]\\
&=[\rho u] + (\rho -u^2\rho^{2-\gamma})[u] -
uv\rho^{2-\gamma}[v],\\
\pt{v}G\pr{U;U^-}&= u\pt{v}\rho [u] + [\rho v]
 + (\rho +v\pt{v}\rho)[v]\\
&=-uv\rho^{2-\gamma}[u] + [\rho v] + (\rho - v^2\rho^{2-\gamma})[v].
\end{split}
\]
Then
\begin{equation*}
\begin{split}
\alpha&=\rho\Big(u-\frac{\rho_\infty}{\rho} + (u-1)\pr{1-u^2\rho^{1-\gamma}}
    - uv^2\rho^{1-\gamma}\Big)
=O(1) \qquad\text{as }\nu\sTo0,\\
\beta&=\rho\big(-(u-1)uv\rho^{1-\gamma} + v + v\pr{1-v^2\rho^{1-\gamma}}\big)
=O(1)\nu^{\frac1{\gamma-1}} \qquad\text{as }\nu\sTo0,
\end{split}
\end{equation*}
where $O(1)$ depends only on $\gamma$ and $b$. Then
\begin{equation}\label{402}
|\frac{\beta}{\alpha}|= O(1)\nu^{\frac1{\gamma-1}}
\qquad\qquad\text{ as }\nu\sTo0.
\end{equation}

Therefore, there exist constants $\nu_0$ and $\varepsilon_0$ such
that, for any $0<\nu\leq\nu_0$ and $0<\varepsilon\leq\varepsilon_0$,
\[
\|(A_0-\hat A, B_0-\hat B)\|_{C_0(\Omega)} +
\norm{\varphi'-\tg\omega_0}_{C^1_{(0)}(\Gamma_0)} + \big
|\frac\alpha\beta-\ctg\omega_1\big|\leq \hat{\epsilon},
\]
where $\hat{A}$ and $\hat{B}$ are the matrices in \eqref{4.16} and
$\hat{\epsilon}$ is the constant in Proposition \ref{le:002}.

If $F\in (W_{(1)}^{0,q}(\Omega))^2$ and $g_j\in
W_{(0)}^{1-1/q,q}(\Gamma_j), j=0,1$, by Proposition \ref{le:002},
there exists a unique solution $\delta U_*\in
(W_{(0)}^{1,q}(\Omega))^2$ to the linearized boundary value problem
\eqref{302}---\eqref{304} such that
\begin{equation}\label{eq:LE}
\norm{\delta U_*}_{W_{(0)}^{1,q}(\Omega)}
\leq K\,\Big(\norm{F}_{W_{(1)}^{0,q}(\Omega)} +
   \sum_{j=0,1}\norm{g_j}_{W_{(0)}^{1-1/q,q}(\Gamma_j)}\Big),
\end{equation}
where $K$ is independent of $(\delta U_*, F, g_0,g_1)$, but depends
only on $b$ and $\gamma$.

\medskip
With the linearized problem, we will start the {\it iteration
scheme} with $F=(f_1,f_2)^\top$ and $g_j, j=0,1,$ that take the
following form:
\begin{equation}\label{305}
\begin{split}
F(\delta U;\psi)\defs&\tilde C(\eta;\psi)\,U - \tilde
A(U;\psi)\,D_2U -
\tilde B(U;\psi)\,D_1U\\
&+A_0\,\pt{\xi}\delta U + B_0\,\pt{\eta}\delta U +
C(\eta)\,\delta U\\
&-\big(A(U)\,\pt{\xi}U + B(U)\,\pt{\eta}U +
C(\eta)\,U\big),\\
g_0(\delta U;\psi):=& u_0(\varphi'(x(\xi,\xi\tg\omega_0))-\tg\omega_0),\\
g_1(\delta U;\psi)\defs&\alpha\,\delta u+\beta\,\delta v - G(U;U^-),
\end{split}
\end{equation}
where $U=U_0+\delta U$. For simplicity, write
$\delta\dot\psi=\dot\psi-\ctg\omega_1$,
$\delta\varphi'=\varphi'-\tg\omega_0$, and
\[\begin{split}
f_1&=\big(\frac{1}{\eta}-\frac{1}{y(\eta;\psi)}\big)v +
\frac{\delta\varphi'}{1 +
\tg\omega_0\,\delta\dot\psi}\Big(\big(1-\frac{u^2}{c^2}\big)D_2u
-\frac{uv}{c^2}D_2v\Big)\\
&\quad+\frac{\delta\dot\psi}{1 +
\tg\omega_0\,\delta\dot\psi}\Big(-\frac{uv}{c^2}D_1u +
\big(1-\frac{v^2}{c^2}\big)D_1v\Big)\\
&\quad + \big(1-\frac{u_0^2}{c_0^2}\big)\pt{\xi}\delta u -
\frac{u_0v_0}{c_0^2}\pr{\pt{\eta}\delta u + \pt{\xi}\delta v} +
\big(1-\frac{v_0^2}{c_0^2}\big)\pt{\eta}\delta v + \frac{\delta
v}{\eta}\\
&\quad -\Big(\big(1-\frac{u^2}{c^2}\big)\pt{\xi}u -
\frac{uv}{c^2}\pr{\pt{\eta}u + \pt{\xi}v} +
\big(1-\frac{v^2}{c^2}\big)\pt{\eta}v + \frac{v}{\eta}\Big),\\
f_2&=\frac{\delta\varphi'}{1 +
\tg\omega_0\,\delta\dot\psi}D_2v-\frac{\delta\dot\psi}{1 +
\tg\omega_0\,\delta\dot\psi}D_1u.
\end{split}\]

\section{Proof of Main Theorem I: Fixed Point of the Iteration Map $\fj$}

In this section, we first prove that there exists a unique fixed
point of the iteration mapping $\fj$ introduced in Section 5.
To achieve this, we prove that $\fj$ is a well-defined, contraction
mapping.

\medskip
We will need the following lemma.

\begin{lemma}\label{le:003} Suppose that $h(0)=0$ and
$\norm{h'(\me^t)}_{L^q(\Real)}<\infty$. Then
\[
\|\frac{h(\me^t)}{\me^t}\|_{L^q(\Real)}\leq
K\norm{h'(\me^t)}_{L^q(\Real)},
\]
that is, $\norm{h}_{W^{0,q}_{(-1)}(\Real_+)}\leq
K\norm{h'}_{W^{0,q}_{(0)}(\Real_+)}$. Moreover, for any constant
$s\neq0$, we have
\[\norm{h(s\me^t)}_{L^q(\Real)}=\norm{h(\me^t)}_{L^q(\Real)}.\]
\end{lemma}

These can be seen by the following direct calculations:
\[\begin{split}
&\int_{-\infty}^\infty\big\|\frac{h(\me^t)}{\me^t}\big\|^q\dif t=
\int_0^\infty\big\|\frac{h(x)}{x}\big\|^q\dif(\ln x)
=\int_0^\infty\frac 1x\big|\int_0^1h'(sx)\dif s\big|^q\dif x\\
&\leq\int_0^\infty\frac 1x\int_0^1\abs{h'(sx)}^q\dif s\dif x
=\int_0^1\int_0^\infty\frac 1x\abs{h'(sx)}^q\dif x\dif s
=\int_0^1\int_0^\infty\frac 1x\abs{h'(x)}^q\dif x\dif s\\
&=\int_0^\infty\frac 1x\abs{h'(sx)}^q\dif x=\int_{-\infty}^\infty
\abs{h'(\me^t)}^q\dif t,
\end{split}\]
and, for any constant $s\ne 0$,
\[\begin{split}
&\int_{-\infty}^\infty\abs{h(s\me^t)}^q\dif t=
\int_0^\infty\abs{h(sx)}^q\dif(\ln x)
=\int_0^\infty\frac 1{sx}\abs{h(sx)}^q\dif (sx)\\
&=\int_0^\infty\abs{h(y)}^q\dif(\ln y) =\int_{-\infty}^\infty
\abs{h(\me^t)}^q\dif t.
\end{split}\]

\subsection{Well-definedness of the iteration mapping $\fj$}

We first show that there exist positive constants $M$ and
$\varepsilon_0$ such that, for any $0<\varepsilon\leq\varepsilon_0$,
$\fj$ is well-defined in $O_{M\varepsilon}$ with the help of
estimate \eqref{eq:LE}.

By Lemma \ref{le:003}, we have
\[\begin{split}
&\big\|\big(\frac 1\eta -
\frac{1}{y(\eta;\psi)}\big)v_0\big\|_{W_{(1)}^{0,q}(\Omega)}\\
&=\big\|v_0\,\frac{1}{y\eta}\big(\varphi(x)
   -x\tg\omega_0+\tg\omega_0(\psi(\eta)-\eta\ctg\omega_1)\big)\big\|_{W_{(1)}^{0,q}(\Omega)}\\
&=O(1)\nu^{\frac1{\gamma-1}}\Big(\big\|\frac{x(\varphi(x)-x\tg\omega_0)}{y\eta
    x}\big\|_{W_{(1)}^{0,q}(\Omega)}
  +\big\|\frac{\tg\omega_0(\psi(\eta)-\eta\ctg\omega_1)}{y\eta}
            \big\|_{W_{(1)}^{0,q}(\Omega)}\Big)\\
&=O(1) K(\omega_0)\nu^{\frac1{\gamma-1}}
  \Big(\norm{\delta\varphi'}_{W_{(0)}^{0,q}(\Real_+)}
        +\|\delta\dot\psi\|_{W_{(0)}^{0,q}(\Real_+)}\Big)\\
&=O(1) K(\omega_0)(1+M_S)\nu^{\frac1{\gamma-1}}\varepsilon.
\end{split}\]
Similarly, we have
\begin{eqnarray*}
&&\big\|\Big(\frac 1\eta - \frac{1}{y(\eta)}\Big)\delta
      v\big\|_{W_{(1)}^{0,q}(\Omega)}
      \leq K M_S\varepsilon\norm{\delta v}_{W_{(1)}^{0,q}(\Omega)}
        \leq K M_SM\varepsilon^2,\\
&&\big\|\frac{\delta\dot\psi}{1+\tg\omega_0\, \delta\dot\psi}D_1u_0
   \big\|_{W_{(1)}^{0,q}(\Omega)}
=\big\|\frac{\delta\dot\psi}{1 + \tg\omega_0\,\delta\dot\psi}\frac
        1r(-\sin\theta+\tg\omega_0\cos\theta)
        \pt{\theta}u_0\big\|_{W_{(1)}^{0,q}(\Omega)}\\
&&\qquad \leq
  K\norm{v_0}_{L^\infty}\|\delta\dot\psi\|_{W_{(0)}^{0,q}(\Real_+)}
   \leq K\,M_S\nu^{\frac1{\gamma-1}}\varepsilon,\\
&& \big\|\frac{\delta\dot\psi}{1 +
    \tg\omega_0\,\delta\dot\psi}D_1(\delta u)\big\|_{W_{(1)}^{0,q}(\Omega)}\leq
    K M_SM\varepsilon^2,\\
&&\big\|\big(1-\frac{u_0^2}{c_0^2}\big)\pt{\xi}u
   -\big(1-\frac{u^2}{c^2}\big)\pt{\xi}u\big\|_{W_{(1)}^{0,q}(\Omega)}
=\big\|\big(\frac{u^2}{c^2}-\frac{u_0^2}{c_0^2}\big)
   \pr{\pt{\xi}u_0+\pt{\xi}\delta u}\|_{W_{(1)}^{0,q}(\Omega)}\\
&&\qquad\leq K M\nu^{\frac2{\gamma-1}}\varepsilon.
\end{eqnarray*}

The other terms in the expression of $F$ can be estimated
analogously. Hence, we have
\begin{equation}\label{406}
\norm{F}_{W_{(1)}^{0,q}(\Omega)}
\leq
K\varepsilon\Big(M_S\nu^{\frac1{\gamma-1}}+M\nu^{\frac2{\gamma-1}}+MM_S\varepsilon\Big).
\end{equation}

It is easy to see that
\begin{equation}\label{eq:ItEg1}
\norm{g_0}_{W_{(0)}^{1-1/q,q}(\Gamma_0)}\leq
K\nu^{\frac1{\gamma-1}}\norm{\delta\varphi'}_{W_{(0)}^{1-1/q,q}(\Gamma_0)}\leq
K\nu^{\frac1{\gamma-1}}\varepsilon.
\end{equation}
Furthermore, we have
\begin{equation}\label{eq:ItEg2-01}
\norm{g_1}_{W_{(0)}^{1-1/q,q}(\Gamma_0)}\leq
K\norm{g_1}_{W_{(0)}^{1,q}(\Omega)},
\end{equation}
and
\[
g_1=\big(G(U;U_0^-)-G(U;U^-)\big) + \big(\alpha\,\delta
u+\beta\,\delta v - G(U;U_0^-)\big).
\]

With a direct calculation, we have
\[\begin{aligned}
&(\pt{u^-}, \pt{v^-})\rho^- = -(\rho^-)^{2-\gamma}(u^-,v^-),\\
&(\partial^2_{u^-}, \partial^2_{v^-})\rho^- =
-(\rho^-)^{2-\gamma}\big(1+(2-\gamma)(\rho^-)^{1-\gamma}(u^-)^2,
1+(2-\gamma)(\rho^-)^{1-\gamma}(v^-)^2\big), \\
&\pt{u^-v^-}\rho^- = (2-\gamma)u^-v^-(\rho^-)^{3-2\gamma}.
\end{aligned}\]
Then
\[\begin{split}
\pt{u^-}G&=-[\rho u] - [u]\big(\rho^-
-(u^-)^2(\rho^-)^{2-\gamma}\big) +
[v]u^-v^-(\rho^-)^{2-\gamma},\\
\pt{v^-}G&= [u]u^-v^-(\rho^-)^{2-\gamma} - [\rho v] - [v]\big(\rho^-
- (v^-)^2(\rho^-)^{2-\gamma}\big),
\end{split}\]
and
\begin{eqnarray*}
&&\partial^2_{u^-}G = -\pr{[u]u^-+[v]v^-}\partial^2_{u^-}\rho^-
  + 2\pr{\rho^- -(u^-)^2(\rho^-)^{2-\gamma}} + 2[u]u^-(\rho^-)^{2-\gamma},\\
&&\partial^2_{v^-}G
  =-\pr{[u]u^-+[v]v^-}\partial^2_{v^-}\rho^- + 2\pr{\rho^-
     -(v^-)^2(\rho^-)^{2-\gamma}} + 2[v]v^-(\rho^-)^{2-\gamma},\\
&&\pt{u^-v^-}G = -\pr{[u]u^-+[v]v^-}\pt{u^-v^-}\rho^- -
2u^-v^-(\rho^-)^{2-\gamma} +
\pr{[u]u^-+[v]v^-}(\rho^-)^{2-\gamma},\\
&&\pt{uu^-}G =-\rho-\rho^- + u^2\rho^{2-\gamma}
+(u^-)^2(\rho^-)^{2-\gamma},\\
&&\pt{vu^-}G = uv\rho^{2-\gamma}
+u^-v^-(\rho^-)^{2-\gamma},\\
&&\pt{uv^-}G = uv\rho^{2-\gamma}
+u^-v^-(\rho^-)^{2-\gamma},\\
&&\pt{vv^-}G =-\rho-\rho^- + v^2\rho^{2-\gamma}
+(v^-)^2(\rho^-)^{2-\gamma}.
\end{eqnarray*}
Since $\varepsilon\ll\nu^{\frac1{\gamma-1}}$, we obtain that, for
any $1<\gamma\leq2$,
\[
\begin{split}
\norm{G(U;U_0^-) - G(U;U^-)}_{W^{1,q}_{(0)}(\Omega)}
=&\big\|\int_0^1\nabla_{U^-}G(U;U_s^-)\dif s\cdot\delta
   U^-\big\|_{W^{1,q}_{(0)}(\Omega)}\\
\leq& K\varepsilon\big(1+\nu^{\frac{4-2\gamma}{\gamma-1}} +
\nu^{\frac{3-2\gamma}{\gamma-1}}M\varepsilon\big)\leq K\varepsilon,
\end{split}\]
where $U_s^-=sU^- + (1-s)U_0^-$ and $K$ is independent of $\nu$ and
$\varepsilon$.

Analogous calculations for $\nabla_UG$, $\nabla_U^2G$, and
$\nabla_U^3G$ yield
\[\begin{split}
\norm{\alpha\,\delta u + \beta\,\delta v -
G(U;U_0^-)}_{W^{1,q}_{(0)}(\Omega)}
&=\big\|\frac12\int_0^1\nabla_U^2G(U_s;U_0^-)\delta U\dif
s\,\delta U\big\|_{W^{1,q}_{(0)}(\Omega)}\\
&\leq KM^2\varepsilon^2,
\end{split}\]
where $U_s=sU+(1-s)U_0$ and $K$ is independent of $\nu$ and
$\varepsilon$.

Hence, by \eqref{eq:ItEg2-01}, we have
\begin{equation}\label{eq:ItEg2-02}
\norm{g_1(\delta U;\psi)}_{W_{(0)}^{1-1/q,q}(\Gamma_1)}\leq
K\varepsilon\big(1+M^2\varepsilon\big),
\end{equation}
where $K$ is independent of $\nu$ and $\varepsilon$.

Therefore, we can choose $\nu_0>0$ and $\varepsilon_0>0$
sufficiently small such that, for any $0<\nu\leq\nu_0$ and
$0<\varepsilon\leq\varepsilon_0$,
\begin{equation}\label{407}
\norm{\delta U_*}_{W^{1,q}_{(0)}(\Omega)}
\leq\hat{K}\varepsilon,
\end{equation}
where $\hat K$ is independent of $\delta U_*$,
$\varepsilon$, and $\nu$, but depends on $\gamma$ and $\omega_0(b)$.

Hereafter, we fix $M=\hat K$. Then the mapping $\fj$ is well-defined
in $O_{M\varepsilon}$.

\subsection{Contraction of the iteration mapping $\fj$} We now show that, for
$\fj(\delta U^j)=\delta U^j_*, j=1,2,$
we can choose $\nu_0>0$ and $\varepsilon_0>0$ sufficiently small
such that, for any $0<\nu\leq\nu_0$ and
$0<\varepsilon\leq\varepsilon_0$,
\begin{equation}\label{408}
\norm{\fj(\delta U^2) - \fj(\delta
U^1)}_{W^{1,q}_{(0)}(\Omega)}\leq\frac12\norm{\delta U^2-\delta
U^1}_{W^{1,q}_{(0)}(\Omega)}.
\end{equation}

Noticing that $\ft(\delta U^j_*)=(F;g_1,g_2)(\delta U^j;\psi),
j=1,2$,
we have
\begin{equation}\label{409}
\begin{split}
&\norm{\delta U^2_{*}-\delta U^1_*}_{W^{1,q}_{(0)}(\Omega)}\\
&\leq K\Big(\norm{F(\delta U^2)-F(\delta
U^1)}_{W^{0,q}_{(1)}(\Omega)} + \sum_{j=0,1}\norm{g_j(\delta
U^2)-g_j(\delta U^1)}_{W^{1-1/q,q}_{(0)}(\Gamma_j)}\Big),
\end{split}
\end{equation}
where $K$ is independent of $\nu$, $\sigma$, $\delta U^j$, and
$\delta U^j_*, j=1,2$, but depends only on $\omega_0(b)$ and
$\gamma$.

Since
\begin{eqnarray*}
&&\big\|\big(\frac1\eta-\frac1{y(\eta;\psi)}\big)\pr{\delta
v^2-\delta v^1}\big\|_{W^{0,q}_{(1)}(\Omega)}\leq
KM_S\varepsilon\norm{\delta v^2 - \delta
v^1}_{W^{1,q}_{(0)}(\Omega)},\\
&&\big\|\frac{\delta\dot\psi}{1 +
\tg\omega_0\,\delta\dot\psi}D_1(\delta u^2-\delta
u^1)\big\|_{W_{(1)}^{0,q}(\Omega)}\leq KM_S\varepsilon\norm{\delta
u^2 - \delta u^1}_{W^{1,q}_{(0)}(\Omega)},
\end{eqnarray*}
\[\begin{split}
&\norm{A_0\,\pt{\xi}\big(\delta U^2 - \delta U^1\big) -
\big(A(U^2)\pt{\xi}U^2
- A(U^1)\pt{\xi}U^1\big)}_{W^{0,q}_{(1)}(\Omega)}\\
&\leq \norm{\big(A_0-A(U^2)\big)\pt{\xi}(\delta U^2 - \delta
U^1)}_{W^{0,q}_{(1)}(\Omega)} + \norm{\big(A(U^2) -
A(U^1)\big)\pt{\xi}U^1
}_{W^{0,q}_{(1)}(\Omega)}\\
&\leq K\big(M\varepsilon+O(1)\nu^{\frac1{\gamma-1}}\big)\norm{\delta
U^2 - \delta U^1}_{W^{1,q}_{(0)}(\Omega)},
\end{split}\]
and analogous estimates for the other terms of $F(\delta
U^2;\psi)-F(\delta U^1;\psi)$, we have
\begin{equation}\label{410}
\norm{F(\delta U^2)-F(\delta U^1)}_{W^{1,q}_{(0)}(\Omega)}\leq
K\varepsilon\big(O(1)\nu^{\frac1{\gamma-1}}+M\varepsilon+M_S\varepsilon\big)
 \norm{\delta U^2 - \delta U^1}_{W^{1,q}_{(0)}(\Omega)}.
\end{equation}

Obviously,
\begin{equation}\label{eq:ItEg1-02}
g_0(\delta U^2;\psi)-g_0(\delta U^1;\psi) = 0.
\end{equation}

Moreover,
\[\begin{split}
&g_1(U^2;\psi) - g_1(U^1;\psi)\\
&=\alpha(\delta u^2-\delta u^1) + \beta(\delta v^2-\delta v^1) -
\big(G(U^2;U^-)-G(U^1;U^-)\big)\\
&=\big(\alpha(\delta u^2-\delta u^1) + \beta(\delta v^2-\delta v^1)
- \pr{G(U^2;U_0^-)-G(U^1;U_0^-)}\big)\\
&\quad+\big(\pr{G(U^2;U_0^-)-G(U^1;U_0^-)} -
    \pr{G(U^2;U^-)-G(U^1;U^-)}\big).
\end{split}\]
Then, an analogous calculation for $g_1$ above in verifying that
$\fj$ is a well-defined mapping in $O_{M\varepsilon}$ yields that
\begin{equation}\label{411}
\norm{g_2(U^2;\psi) - g_2(U^1;\psi)}_{W^{1-1/q,q}_{(0)}(\Real_+)}
\leq KM\varepsilon\norm{\delta U^2 - \delta
U^1}_{W^{1,q}_{(0)}(\Omega)}.
\end{equation}

Choose $\nu_0>0$ and $\varepsilon_0>0$ sufficiently small. Then, for
any $0<\varepsilon\leq\varepsilon_0$ and $0<\nu\leq\nu_0$, estimates
\eqref{409}--\eqref{411}
imply that \eqref{408} holds, that is, $\fj$ is a contraction
mapping in $O_{M\varepsilon}$.

\section{Proof of Main Theorem II: Fixed Point of the Iteration Map $\fj_S$}

In this section, we prove that there exists a unique fixed point of
the iteration mapping $\fj_S$ introduced in Section 5 by showing
that $\fj_S$ is a well-defined, contraction mapping, which completes
the proof of the main theorem.

\subsection{Well-definedness of the iteration mapping $\fj_S$}

Let $\fj_S(\psi)=\psi_*$. Write
\[
\Psi(U;U^-;\psi)=-\frac{[v](1-\tg\omega_0\ctg\omega_1)}
{[u]+\varphi'(x)[v]}.
\]
Then
\[\begin{aligned}
&\pt{u}\Psi=\frac{[v](1-\tg\omega_0\ctg\omega_1)}{([u]+\varphi'(x)[v])^2},&&
\pt{u^-}\Psi=-\frac{[v](1-\tg\omega_0\ctg\omega_1)}{([u]+\varphi'(x)[v])^2},\\
&\pt{v}\Psi=-\frac{[u](1-\tg\omega_0\ctg\omega_1)}{([u]+\varphi'(x)[v])^2},&&
\pt{v^-}\Psi=\frac{[u](1-\tg\omega_0\ctg\omega_1)}{([u]+\varphi'(x)[v])^2}.
\end{aligned}\]

Thus, by \eqref{eq:UpdateApS}, we obtain
\begin{equation}\label{412}
\norm{\dot\psi_*-\ctg\omega_1}_{S}\leq\tilde{K}M\varepsilon,
\end{equation}
where $\tilde K$ is a constant independent of $\nu$ and
$\varepsilon$.

We choose $M_S=\tilde KM$ hereafter. Then $\fj_S$ is well-defined in
$\Sigma_{M_S\varepsilon}$ in the case that the positive constants
$\nu$ and $\varepsilon$ are sufficiently small. To complete the
proof,
it suffices to verify that $\fj_S$ is a contraction mapping in
$\Sigma_{M_S\varepsilon}$.

\subsection{Contraction of the iteration mapping $\fj_S$}

Let $\fj_S(\psi^j)=\psi^j_*, j=1,2$.
Then we have
\[
\lb{\begin{aligned}&\ft(\delta U_j)=(F;g_0,g_1)(\delta U_j;\psi^j), \qquad j=1,2,\\
&\dot\psi^j_*=\Psi(U_j;U^-;\psi^j).
\end{aligned}}
\]
Thus, we obtain
\begin{equation}\label{413}
\begin{split}
&\norm{\delta U_2-\delta U_1}_{W^{1,q}_{(0)}(\Omega)}\\
&\leq K\Big(\norm{F(\delta U_2;\psi^2)-F(\delta
U_1;\psi^1)}_{W^{0,q}_{(1)}(\Omega)}\\
&\qquad\,\,\, + \sum_{j=0,1}\norm{g_j(\delta U_2;\psi^2)-g_j(\delta
U_1;\psi^1)}_{W^{1-1/q,q}_{(0)}(\Gamma_j)}\Big),
\end{split}
\end{equation}
where $K$ is independent of $\delta U_j$ and $\psi^j, j=1,2$, but
depends only on $\omega_0(b)$ and $\gamma$.

Since
\[\begin{split}
J_1\defs&\Big(\frac1\eta-\frac1{y(\eta;\psi^2)}\Big)v_2 -
\Big(\frac1\eta-\frac1{y(\eta;\psi^1)}\Big)v_1\\
=&\Big(\frac1{y(\eta;\psi^1)}-\frac1{y(\eta;\psi^2)}\Big)v_2 +
\Big(\frac1\eta-\frac1{y(\eta;\psi^1)}\Big)(\delta v_2 - \delta
v_1),
\end{split}\]
and
\[\frac1{y(\eta;\psi^1)}-\frac1{y(\eta;\psi^2)} =
\frac{\tg\omega_0(\psi^2-\psi^1)}{y(\eta;\psi^1)y(\eta;\psi^2)},\]
we have
\[
\norm{J_1}_{W^{0,q}_{(1)}(\Omega)}\leq
K\big(\nu^{\frac1{\gamma-1}}+M\varepsilon\big)
\|\dot\psi^2-\dot\psi^1\|_{W^{0,q}_{(0)}(\Gamma_1)}  +
KM_S\varepsilon\norm{\delta v_2-\delta v_1}_{W^{1,q}_{(0)}(\Omega)}.
\]
Set
\[\begin{split}
J_2\defs&-\frac{\delta\dot\psi^2}{1 + \tg\omega_0\,
\delta\dot\psi^2}D_1u_2
+ \frac{\delta\dot\psi^1}{1 + \tg\omega_0\,\delta\dot\psi^1}D_1u_1\\
=&\Big(\frac{\delta\dot\psi^1}{1 + \tg\omega_0\,\delta\dot\psi^1} -
\frac{\delta\dot\psi^2}{1 +
\tg\omega_0\,\delta\dot\psi^2}\Big)D_1u_2 +
\frac{\delta\dot\psi^1}{1 + \tg\omega_0\,\delta\dot\psi^1}
D_1\pr{\delta u_1-\delta u_2}.
\end{split}\]
An analogous calculation yields
\[\begin{split}
\norm{J_2}_{W^{0,q}_{(1)}(\Omega)}
\leq&K\big(\nu^{\frac1{\gamma-1}}+M\varepsilon\big)
\|\dot\psi^2-\dot\psi^1\|_{\Gamma_1} +KM_S\varepsilon\norm{\delta
u_2-\delta u_1}_{W^{1,q}_{(0)}(\Omega)}.
\end{split}\]
Set
\[\begin{split}
J_3\defs&\big(1-\frac{u_0^2}{c_0^2}\big)\pt{\xi}(\delta u_2-\delta
u_1) -\Big(\big(1-\frac{u_2^2}{c_2^2}\big)\pt{\xi}u_2 -
\big(1-\frac{u_1^2}{c_1^2}\big)\pt{\xi}u_1\Big)\\
=&\big(\frac{u_1^2}{c_1^2}-\frac{u_0^2}{c_0^2}\big)\pt{\xi}(\delta
u_2 - \delta u_1) +
\big(\frac{u_2^2}{c_2^2}-\frac{u_1^2}{c_1^2}\big)\pt{\xi}u_2.
\end{split}\]
Then, as the calculation for $\fj$, we have
\[
\norm{J_3}_{W^{0,q}_{(1)}(\Omega)}\leq
K\big(\nu^{\frac1{\gamma-1}}+M\varepsilon\big) \norm{\delta
u_2-\delta u_1}_{W^{1,q}_{(0)}(\Omega)}.
\]

Analogous calculation for the other terms of $F(\delta
U_2;\psi^2)-F(\delta U_1;\psi^1)$ finally leads to
\begin{equation}\label{414}
\begin{split}
&\norm{F(\delta U_2;\psi^2)-F(\delta U_1;\psi^1)}_{W^{0,q}_{(1)}(\Omega)}\\
&\leq K \big(\nu^{\frac1{\gamma-1}}+M\varepsilon\big)
\|\dot\psi^2-\dot\psi^1\|_{\Gamma_1}
+K\big(M_S\varepsilon+M\varepsilon+\nu^{\frac1{\gamma-1}}\big)\norm{\delta
U_2-\delta U_1}_{W^{1,q}_{(0)}(\Omega)}.
\end{split}
\end{equation}

Since
\[
\begin{split}
g_0(\delta U_2;\psi^2) - g_0(\delta U_1;\psi^1)
=
&u_0\big(\varphi'(x(\xi,\xi\tg\omega_0;\psi^2))
 -\varphi'(x(\xi,\xi\tg\omega_0;\psi^1))\big)\\
=&u_0\varphi''\,\big(x(\xi;\psi^2) -
x(\xi;\psi^1)\big)\\
=&u_0\varphi''\,\big(\psi^2(\xi\tg\omega_0) -
\psi^1(\xi\tg\omega_0)\big),
\end{split}
\]
and
\[
\frac{\dif}{\dif\xi}\varphi'(x(\xi,\xi\tg\omega_0;\psi^j)) =
\varphi''(x)\,\big(\pt{\xi}x + \tg\omega_0\pt{\eta}x\big) =
\varphi''(x)\,\big(1+\tg\omega_0\delta\dot\psi^j\big),
\]
we have
\begin{equation}\label{eq:ItEg1-03}
\norm{g_0(\delta U_2;\psi^2) - g_0(\delta
U_1;\psi^1)}_{W^{1-1/q,q}_{(0)}(\Gamma_0)}\leq
K\nu^{\frac1{\gamma-1}}\varepsilon\|\dot\psi^2 -
\dot\psi^1\|_{\Gamma_1},
\end{equation}
where $K$ depends only on $\omega_0(b)$ and $\gamma$, but
independent of $\psi^j$, $\nu_0$, and $\varepsilon_0$.

Furthermore, we have
\[\begin{split}
&g_1(\delta U_2;\psi^2)-g_1(\delta U_1;\psi^1) \\
=& \alpha(\delta u_2 - \delta u_1) + \beta(\delta v_2 - \delta v_1)
- \big(G(U_2;U^-(\Gamma_1;\psi^2)) - G(U_1;U^-(\Gamma_1;\psi^1))\big)\\
=& \alpha(\delta u_2 - \delta u_1) + \beta(\delta v_2 - \delta v_1)
- \big(G(U_2;U^-(\Gamma_1;\psi^2)) - G(U_1;U^-(\Gamma_1;\psi^2))\big)\\
&\, + \big(G(U_1;U^-(\Gamma_1;\psi^2)) -
G(U_1;U^-(\Gamma_1;\psi^1))\big),
\end{split}\]
where $U^-(\Gamma_1;\psi^j)=U^-(\eta\ctg\omega_1,\eta;\psi^j),
j=1,2$. Notice that
\[
\begin{split}
&G(U_1;U^-(\Gamma_1;\psi^2)) - G(U_1;U^-(\Gamma_1;\psi^1)) \\
&=\int_0^1\big(\pt{u^-}G(U_1;U^-_s)\pr{\delta u^-(\Gamma_1;\psi^2) -
\delta u^-(\Gamma_1;\psi^1)} \\
&\qquad\quad\,\, + \pt{v^-}G(U_1;U^-_s)\pr{\delta
v^-(\Gamma_1;\psi^2) -
  \delta v^-(\Gamma_1;\psi^1)}\big)\dif s,
\end{split}\]
where $U^-_s = sU^-(\Gamma_1;\psi^2) + (1-s)U^-(\Gamma_1;\psi^1)$,
and
\[\begin{split}
\delta U^-(\Gamma_1;\psi^2) - \delta U^-(\Gamma_1;\psi^1)
=&\, \delta U^-(\psi^2(\eta),\eta) - \delta U^-(\psi^1(\eta),\eta)\\
=&\int_0^1\pt{\xi}U^-(s\psi^2+(1-s)\psi^1,\eta)\dif
s\,\pr{\psi^2-\psi^1}.
\end{split}
\]

Then an analogous calculation as for $\fj$ yields
\begin{equation}\label{415}
\begin{split}
&\norm{g_1(\delta U_2;\psi^2)-g_1(\delta
U_1;\psi^1)}_{W^{1-1/q,q}_{(0)}(\Gamma_1)}\\
&\leq K_1M\varepsilon\norm{\delta U_2-\delta
U_1}_{W^{1,q}_{(0)}(\Omega)} + K_2\varepsilon\|\dot\psi^2 -
\dot\psi^1\|_{\Gamma_1},
\end{split}
\end{equation}
where $K_1$ and $K_2$ depend on $\omega_0(b)$ and $\gamma$.

Then, by \eqref{413}--\eqref{415},
we have
\begin{equation}\label{416}
\begin{split}
&\norm{\delta U_2-\delta U_1}_{W^{1,q}_{(0)}(\Omega)}\\
&\leq K\big(\nu^{\frac1{\gamma-1}}+M\varepsilon\big)
\|\dot\psi^2-\dot\psi^1|_{\Gamma_1}
+K\big(M_S\varepsilon+M\varepsilon+\nu^{\frac1{\gamma-1}}\big)\norm{\delta
U_2-\delta U_1}_{W^{1,q}_{(0)}(\Omega)}.
\end{split}
\end{equation}

Choose $\nu_0$ and $\varepsilon_0$ sufficiently small. Then, for any
$0<\nu\leq\nu_0$ and $0<\varepsilon\leq\varepsilon_0$, we have
\begin{equation}\label{417}
\begin{split}
&\norm{\delta U_2-\delta U_1}_{W^{1,q}_{(0)}(\Omega)}\leq
K\big(\nu^{\frac1{\gamma-1}}+M\varepsilon\big)
\|\dot\psi^2-\dot\psi^1\|_{\Gamma_1}.
\end{split}
\end{equation}
Thus,
\[\begin{split}
\|\dot\psi^2_*-\dot\psi^1_*\|_{\Gamma_1}&=
\|\dot\psi^2_*-\dot\psi^1_*\|_{W^{0,q}_{(0)}(\Real_+)} +
\|\dot\psi^2_*-\dot\psi^1_*\|_{C^{0}(\Real_+)}\\
&\leq K\norm{\delta U_2-\delta U_1}_{W^{1,q}_{(0)}(\Omega)} +
K\norm{\delta U^-(\Gamma_1;\psi^2) - \delta U^-(\Gamma_1;\psi^1)}_{\Gamma_1}\\
&\leq K\big(\nu^{\frac1{\gamma-1}}+M\varepsilon+\varepsilon\big)
\|\dot\psi^2-\dot\psi^1\|_{\Gamma_1}\\
&\leq \frac12\|\dot\psi^2-\dot\psi^1\|_{\Gamma_1},
\end{split}\]
where, for the last inequality, we have again chosen $\nu_0$ and
$\varepsilon_0$ to be sufficiently small.

This implies that $\fj_S$ is a contraction mapping so that it has a
unique fix point in $\Sigma_{M_S\varepsilon}$, which completes the
proof of Theorem \ref{th:FixB}.

\appendix

\section*{Appendix: A Fredholm-type Theorem}
\stepcounter{section}

To be self-contained, in this appendix, we give a proof for a
Fredholm-type theorem, Theorem A.1, a special case of Theorem 4.1 in
Maz'ya-Plamenevski\v{i} \cite{MP}, following their ideas. Consider
the boundary value problem of an elliptic equation of second-order
in an infinite strip $\G:=\set{(t,x):\ x\in I:=(x_0,x_1),t\in\Real}$
with boundaries $\Sigma_0=\set{x=x_0}$ and $\Sigma_1=\set{x=x_1}$:
\begin{align}
&L\varphi:=\pt{tt}\varphi + \pt{xx}\varphi + \pt{t}\varphi +
a(x)\pt{x}\varphi = f &&\text{in }\G,\label{Eq}\\
&B_0\varphi:=\pt{x}\varphi=g_0&&\text{on }\Sigma_0,\label{BdC0}\\
&B_0\varphi:=\pt{x}\varphi=g_1&&\text{on }\Sigma_1,\label{BdC1}
\end{align}
where $a(x)\in C^1(\bar I)$, $f\in W_{(-1)}^{0,q}(\G)$, and $g_j\in
W_{(-1)}^{1-1/q,q}(\Real), j=0,1$. We assume $q>2$ since only this
case is really used in this paper. Obviously, the operator
$(L;B_0,B_1)$ of the boundary value problem \eqref{Eq}--\eqref{BdC1}
acts continuously from the space $W_{(-1)}^{2,q}(\G)$ to
$W_{(-1)}^{0,q}(\G)\times\big(W_{(-1)}^{1-1/q,q}(\Real)\big)^2$.

Consider the boundary value problem with a complex parameter
$\lambda$ on the interval $I$:
\begin{align}
&\pt{xx}\varphi  +
a(x)\pt{x}\varphi + (-\lambda^2+\mi\lambda)\varphi = f &&\text{in }I,\label{EqP}\\
&\pt{x}\varphi=g_0&&x=x_0,\label{BdC0P}\\
&\pt{x}\varphi=g_1&&x=x_1.\label{BdC1P}
\end{align}
For all $\lambda$, with the exception of certain isolated points,
\eqref{EqP}--\eqref{BdC1P} has a unique solution $\varphi\in
W^{2,p}$. The exception isolated points of $\lambda$ are called
spectrum of problem \eqref{EqP}--\eqref{BdC1P}.

Then we have
\begin{theorem}\label{th:MP}
If the line ${\rm Im}\, \lambda=\beta$ does not contain the
eigenvalues of problem \eqref{EqP}--\eqref{BdC1P}, then the operator
$(L;B_0,B_1)$ of problem \eqref{Eq}--\eqref{BdC1} realizes an
isomorphism:
\[
W_{(\beta)}^{2,p}(\G)\approx
W_{(\beta)}^{0,q}(\G)\times\big(W_{(\beta)}^{1-1/q,q}(\Real)\big)^2.
\]
Moreover, the solution $\varphi\in W_{(\beta)}^{2,p}(\G)$ of
\eqref{Eq}--\eqref{BdC1} satisfies the estimate:
\begin{equation}\label{MPEsti}
\norm{\varphi}_{W_{(\beta)}^{2,p}(\G)}\leq
K\Big(\norm{f}_{W_{(\beta)}^{0,q}(\G)} +
\sum_{j=0,1}\norm{g_j}_{W_{(\beta)}^{1-1/q,q}(\Real)}\Big).
\end{equation}
\end{theorem}
\begin{rem}
In the case $p=2$, this assertion is well-known (cf. \cite{Kon}). In
this case, a solution in the class $W_{(\beta)}^{2,2}(\G)$ can be
represented in the form
\begin{equation}\label{MPSolu}
\varphi(t)=\frac{1}{\sqrt{2\pi}}\int_{\text{Im}\lambda=\beta}\me^{-\mi\lambda
t}R(\lambda)\ff_{t\sTo\lambda}\set{f;g_0,g_1}\dif\lambda,
\end{equation}
where $R(\lambda)$ denotes the inverse operator of problem
\eqref{EqP}--\eqref{BdC1P} and $\ff_{t\sTo\lambda}$ is the Fourier
transform with respect to the $t$-variable into the
$\lambda$-variable. If it is additionally assumed that $f\in
W_{(\beta_1)}^{0,2}(\G),g_j\in W_{(\beta_1)}^{1-1/2,2}(\Real)$ and
that, in the closed strip between the lines Im$\,\lambda=\beta$ and
Im$\,\lambda=\beta_1$, there are no points of the spectrum of
\eqref{EqP}--\eqref{BdC1P},  then the function $\varphi$ defined by
\eqref{MPSolu} belongs to $W_{(\beta_1)}^{2,2}(\G)$, and
\begin{equation}\label{MPEsti1}
\norm{\varphi}_{W_{(\beta_1)}^{2,2}(\G)}\leq
K\Big(\norm{f}_{W_{(\beta_1)}^{0,2}(\G)} +
\sum_{j=0,1}\norm{g_j}_{W_{(\beta_1)}^{1/2,2}(\Real)}\Big).
\end{equation}
\end{rem}

To prove Theorem \ref{th:MP}, we need two lemmas, which are all in
\cite{MP}.

Let $\fa_0$, $\fa_1$, and $\fa_2$ be Banach spaces of functions on
$\Real$, in each of which multiplication by scalar functions in
$C_{c}^\infty(\Real)$ is defined. Let
$\set{\zeta_k}_{-\infty}^\infty$ be a partition of unity on $\Real$
subordinate to the covering of $\Real$ by the intervals
$(k-1)\delta<t<(k+1)\delta$, where $\delta$ is a fixed positive
number and $\zeta_k\in C^\infty(\Real)$. Suppose that the norms
$\norm{\cdot}_j$ in the spaces $\fa_j, j=0,1,2$, possess the
following properties: For $p\in[1,\infty]$,
\begin{align}
&C_1\norm{u}_0\leq\Big(\sum_{k=-\infty}^\infty\norm{\zeta_ku}_0^p\Big)^{1/p}\leq
C_2\norm{u}_0,\label{MPNormP1}\\
&\norm{v}_1\geq C\Big(\sum_{k=-\infty}^\infty\norm{\zeta_kv}_1^p\Big)^{1/p},\label{MPNormP2}\\
&\norm{w}_2\leq
C\Big(\sum_{k=-\infty}^\infty\norm{\zeta_kw}_2^p\Big)^{1/p}.\label{MPNormP3}
\end{align}

\begin{lemma}\label{le:MP1}
Let $\fp:\ \fa_1\sTo\fa_0$ be a linear operator defined on the
functions with compact support and such that, for some
$\varepsilon>0$ and any integers $m$ and $k$,
\begin{equation}\label{MPLOP}
\norm{\zeta_k\fp(\zeta_m v)}_0\leq
C\me^{-\abs{m-k}\varepsilon}\norm{\zeta_m v}_1 \qquad\mbox{for any}
\,\, v \in \fa_1.
\end{equation}
Then

{\rm (i)} For all $v\in\fa_1$ with compact support,
\begin{equation}\label{MPLOP1}
\norm{\fp v}_0\leq C\norm{v}_1,
\end{equation}
where the constant $C$ does not depend on $v$.

{\rm (ii)} Let $\fa_2\subset\fa_0$.
Suppose
further that, for all functions $v$ in $\fa_1$ with compact support
on $\Real$,
\begin{equation}\label{MPLOP2}
\norm{\zeta_k\fp v}_2\leq C\big(\norm{\sigma_kv}_1 +
\norm{\sigma_k\fp v}_0\big),
\end{equation}
where $\sigma_k=\zeta_{k-1}+\zeta_k+\zeta_{k+1}$, $k=0,\pm1, \dots$.
Then
\begin{equation}\label{MPLOP3}
\norm{\fp v}_2\leq C\norm{v}_1.
\end{equation}
\end{lemma}

\begin{proof}
According to \eqref{MPNormP1} and \eqref{MPLOP}, we have
\[\begin{split}
&\norm{\fp v}_0=\big\|\fp\big(\sum_{m=-\infty}^\infty\zeta_m
v\big)\big\|_0\leq C\Big(\sum_{k=-\infty}^\infty
\big\|\sum_{m=-\infty}^\infty\zeta_k\fp(\zeta_mv)\big\|_0^p\Big)^{1/p}\\
\leq&C\Big(\sum_{k=-\infty}^\infty\big(
\sum_{m=-\infty}^\infty\norm{\zeta_k\fp(\zeta_mv)\big)}_0^p\Big)^{1/p}\leq
C\Big(\sum_{k=-\infty}^\infty\big(
\sum_{m=-\infty}^\infty\me^{-\abs{m-k}\varepsilon}\norm{\zeta_mv}_1\big)^p\Big)^{1/p}.
\end{split}\]
Since the operator of discrete convolution with kernel
$\set{\me^{-l\varepsilon}}_{l=-\infty}^\infty$ acts continuously in
$l_p$, it follows that
\[
\norm{\fp v}_0\leq
C\Big(\sum_{m=-\infty}^\infty\norm{\zeta_mv}_1^p\Big)^{1/p}.
\]
The last inequality, together with \eqref{MPNormP2}, leads to
\eqref{MPLOP1}.

Furthermore, by \eqref{MPNormP3} and \eqref{MPLOP2},
\[
\norm{\fp v}_2\leq C\Big(\sum_k\norm{\zeta_k\fp
v}_2^p\Big)^{1/p}\leq C\Big(\sum_k\norm{\sigma_kv}_1^p\Big)^{1/p} +
C\Big(\sum_k\norm{\sigma_k\fp v}_0^p\Big)^{1/p}.
\]
Using the definition of $\sigma_k$, \eqref{MPNormP1}, and
\eqref{MPNormP2}, we obtain
\[
\norm{\fp v}_2\leq C\Big(\norm{v}_1^p+\norm{\fp v}_0^p\Big)^{1/p}.
\]
Then we apply \eqref{MPLOP1} to arrive at the result.
\end{proof}

\begin{lemma}\label{le:MP3}
Suppose the supports of the functions $f$ and $g_j, j=0,1,$ are
concentrated on the set $\set{(t,x)\in\G:m-1<t<m+1}$ ($m$ an
integer), and $f\in W^{0,p}(\G)\cap W^{0,2}(\G)$, $g_j\in
W^{1-1/p,p}(\Real)\cap W^{1-1/2,2}(\Real)$, $j=0,1$, for $p>2$. If
the line {\rm Im }$\lambda=\beta$ does not contains the eigenvalues
of problem \eqref{EqP}--\eqref{BdC1P}, then the solution $\varphi\in
W_{(\beta)}^{2,2}(\G)$ of problem \eqref{Eq}--\eqref{BdC1} satisfies
the estimate
\begin{equation}\label{MPEsti2}
\norm{\me^{\beta t}\zeta_l\varphi}_{L^p(\G)}\leq
C\me^{-\abs{m-l}\varepsilon}\Big(\norm{f}_{W_{(\beta)}^{0,p}(\G)} +
\sum_{j=0,1}\norm{g_j}_{W_{(\beta)}^{1-1/p,p}(\Real)}\Big),
\end{equation}
where $\varepsilon$ is a positive number and
$\set{\zeta_l}_{-\infty}^\infty$ is a partition of unity on $\Real$
subordinate to the covering of $\Real$ by the intervals $l-1<t<l+1$.
\end{lemma}

\begin{proof}
Denote by $M_{p,\beta}$ the term in the parenthesis of the
right-hand side of \eqref{MPEsti2}. Using \eqref{MPEsti1} for any
$\beta_1\in(\beta-\varepsilon,\beta+\varepsilon)$, we obtain
\begin{equation}\label{MPEstiloc1}
\begin{split}
\Big(\int_{l-2}^{l+2}\norm{\me^{\beta t}\varphi}^2_{L^2(I)}\dif
t\Big)^{1/2}\leq&C\me^{(\beta-\beta_1)l}M_{2,\beta_1}\leq
C\me^{(\beta-\beta_1)l}M_{p,\beta_1}\\
\leq&C\me^{(\beta-\beta_1)(l-m)}M_{p,\beta}.
\end{split}
\end{equation}
Applying results in \cite{ADN}, we find that the solution
$\varphi\in W_{loc}^{2,p}(\G)$ of \eqref{Eq}--\eqref{BdC1} has the
following local estimate:
\begin{equation}\label{MPEstiloc2}
\norm{\eta_1\varphi}_{W^{2,p}(\G)}\leq
C\Big(\norm{\eta_2f}_{W^{0,p}(\G)} +
\sum_{j=0,1}\norm{\eta_2g_j}_{W^{1-1/p,p}(\Real)} +
\norm{\eta_2\varphi}_{L^2(\G)}\Big),
\end{equation}
where $\eta_s(t)=\eta(t/s)$, $\eta\in C_c^\infty(-1,1)$, and
$\eta(t)=1$ for $\abs{t}<1/2$.

In the case $\abs{m-l}<2$, the local estimate \eqref{MPEstiloc2}
leads to the estimate
\begin{equation}\label{MPEstiloc3}
\norm{\me^{\beta t}\zeta_l\varphi}_{L^p(\G)}\leq
C\me^{(\beta-\beta_1)(l-m)}M_{p,\beta} +
C\Big(\int_{l-2}^{l+2}\norm{\me^{\beta t}\varphi}^2_{L^2(I)}\dif
t\Big)^{1/2}.
\end{equation}

If $\abs{m-l}\geq2$, then, by \eqref{MPEstiloc2}, the last
inequality remains valid even without the first term on the
right-hand side. Combining \eqref{MPEstiloc1} with
\eqref{MPEstiloc3}, we obtain
\begin{equation}\label{MPEstiloc4}
\norm{\me^{\beta t}\zeta_l\varphi}_{L^p{(\G)}}\leq
C\me^{(\beta-\beta_1)(l-m)}M_{p,\beta}.
\end{equation}
Setting $\beta_1=\beta+\varepsilon$ for $m<l$ and
$\beta_1=\beta-\varepsilon$ for $m\geq l$, we arrive at
\eqref{MPEsti2}.
\end{proof}

Now we prove Theorem \ref{th:MP}.
\begin{proof}[Proof of Theorem {\rm \ref{th:MP}}]
\emph{Existence}. It suffices to prove \eqref{MPEsti} for a solution
$\varphi\in W_{(\beta)}^{2,p}(\G)\cap W_{(\beta)}^{2,2}(\G)$. Let
$\fp$ be the inverse operator of problem \eqref{Eq}--\eqref{BdC1}
defined by \eqref{MPSolu} on the space
$W_{(\beta)}^{2,2}(\G)\times\big(W_{(\beta)}^{1/2,2}(\Real)\big)^2$.
We set
\[\begin{split}
&\norm{u}_{\fa_0}=\norm{\me^{\beta t}u}_{L^p(\G)},\\
&\norm{\set{f;g_0,g_1}}_{\fa_1}=\norm{f}_{W_{(\beta)}^{0,p}(\G)} +
\sum_{j=0,1}\norm{g_j}_{W_{(\beta)}^{1-1/p,p}(\Real)},\\
&\norm{u}_{\fa_2}=\norm{u}_{W_{(\beta)}^{2,p}(\G)}.
\end{split}\]
Lemma \ref{le:MP3} and \eqref{MPEstiloc2} ensure that the hypotheses
of Lemma \ref{le:MP1} are satisfied. Therefore, for the solution
$\varphi=\fp\set{f;g_0,g_1}$, we have \eqref{MPLOP3} or,
equivalently, \eqref{MPEsti}.

\emph{Uniqueness}. Let $\varphi$ be a solution of the homogeneous
problem \eqref{Eq}--\eqref{BdC1} in $W_{(\beta)}^{2,p}(\G)$. We set
$\G_s:=\set{(x,t)\in\G:s<\abs{t}<s+1}$, $s=1,2$,\ldots, and
introduce the sequence of functions $\psi_s\in C_c^\infty(\Real)$,
$\psi_s=1$ for $\abs{t}\leq s$, $\psi_s=0$ for $\abs{t}>1$, and
$\abs{\pt{t}^j\psi_s(t)}\leq C$, $j=1,2$, for some constant
$C<\infty$. Then the function $\psi_s\varphi$ satisfies
\eqref{Eq}--\eqref{BdC1}, where $f$ and $g_j$ are functions
concentrated in $\G_s$. Since $\varphi$ is a solution of the
homogeneous problem, by the local estimate \eqref{MPEstiloc2}, we
have $\varphi\in W_{loc}^{2.2}(\G)$ and
\[
\norm{f}_{W_{(\beta)}^{0,2}(\G)} +
\sum_{j=0,1}\norm{g_j}_{W_{(\beta)}^{1/2,2}(\Real)}\leq
C\norm{\me^{\beta t}\varphi}_{L^p(\G_{s-1}\cup\G_{s}\cup\G_{s+1})}.
\]
From this and  estimate \eqref{MPEsti1} for $\psi_s\varphi$ we
obtain
\[
\norm{\psi_s\varphi}_{W_{(\beta)}^{2,2}(\G)}\leq C\norm{\me^{\beta
t}\varphi}_{L^p(\G_{s-1}\cup\G_{s}\cup\G_{s+1})}.
\]
Since the right side of this inequality tends to zero as
$s\sTo\infty$, it follows that $\varphi=0$. The theorem is proved.
\end{proof}

\bigskip
\noindent {\bf Acknowledgments.} Gui-Qiang Chen's research was
supported in part by the National Science Foundation under Grants
DMS-0505473 and DMS-0244473, by the NSFC under a joint project Grant
10728101, and by an Alexander von Humboldt Foundation Fellowship.
Beixiang Fang's research was supported in part by the Tian-Yuan
Mathematical Foundation of National Natural Science Foundation of
China under Grant 10626035,  a key project of NSFC under Grant
10531020, a joint project from the NSAF of China under Grant
10676020, and the NSF-NSFC under Grant DMS-0720925 ``U.S.-China CMR:
Multidimensional Problems in Nonlinear Conservation Laws and Related
Applied Partial Differential Equations''.

\end{document}